\documentclass[a4paper,11pt]{article}
\usepackage{amsmath, amssymb, graphicx, parskip, amsthm, fancyhdr, dsfont, hyperref, bm, mleftright, fullpage, tikz}
\usetikzlibrary{positioning}

\newtheorem{thm}{Theorem}
\newtheorem{prop}{Proposition}
\newtheorem{lem}{Lemma}

\theoremstyle{definition}

\newtheorem{rem}{Remark}

\newcommand{\bbP}{\mathbb{P}}
\newcommand{\bbN}{\mathbb{N}}
\newcommand{\bbE}{\mathbb{E}}
\newcommand{\bbR}{\mathbb{R}}
\newcommand{\bfX}{\mathbf{X}}
\newcommand{\rmd}{\mathrm{d}}
\newcommand{\calX}{\mathcal{X}}
\newcommand{\bfk}{\mathbf{k}}
\newcommand{\bfl}{\bm{\ell}}
\newcommand{\bfe}{\mathbf{e}}
\newcommand{\calP}{\mathcal{P}}

\title{Genealogical processes of non-neutral population models under rapid mutation}
\author{Jere Koskela \\
	\texttt{jere.koskela@newcastle.ac.uk} \\
	\small School of Mathematics, Statistics and Physics,\\
	\small Newcastle University \\
	\small \& Department of Statistics, \\
	\small University of Warwick \\
	\and
	Paul A.~Jenkins \\
	\texttt{p.jenkins@warwick.ac.uk} \\
	\small Department of Statistics \& \\ 
	\small Department of Computer Science, \\
	\small University of Warwick
	\and
	Adam M.~Johansen \\
	\texttt{a.m.johansen@warwick.ac.uk} \\
	\small Department of Statistics, \\
	\small University of Warwick \\
	\and
	Dario Span\`{o} \\
	\texttt{d.spano@warwick.ac.uk} \\
	\small Department of Statistics, \\
	\small University of Warwick
}
\date{\today}

\begin{document}

\maketitle

\begin{abstract}
We show that genealogical trees arising from a broad class of non-neutral models of population evolution converge to the Kingman coalescent under a suitable rescaling of time.
As well as non-neutral biological evolution, our results apply to genetic algorithms encompassing the prominent class of sequential Monte Carlo (SMC) methods.
The time rescaling we need differs slightly from that used in classical results for convergence to the Kingman coalescent, which has implications for the performance of different resampling schemes in SMC algorithms.
In addition, our work substantially simplifies earlier proofs of convergence to the Kingman coalescent, and corrects an error common to several earlier results.
\end{abstract}



\section{Introduction}

A collection of articles by the present authors and Suzie Brown have focused on scaling limits of genealogical processes for interacting particle systems describing a class of genetic algorithms, as well as biological evolution \cite{brown:2021, brownetal:2021, brownetal:2023, koskela2020annals}.
In a recent note, Sylvain Rubenthaler showed that a strong independence assumption, on which our argument relied, ruled out several interesting applications of our result \cite{rubenthaler:2023+}.
His counterexample is correct, but in this article we show that our argument can be adapted to overcome the issue.
Our proof is based on considering a quenched particle system in which the genealogical process simplifies to a collection of inhomogeneous coalescing random walks, and is substantially simpler and shorter than earlier proofs of similar results in \cite{brown:2021, brownetal:2021, brownetal:2023, koskela2020annals}. 
We also identify a natural timescale for coalescence of nonneutral particle systems (see \eqref{kingman_timescale}), and show in Section \ref{merger_ordering} that it is subtly but materially different from timescales which have been used to obtain similar scaling limits in earlier works.
From the point of view of sequential Monte Carlo (SMC)---a prominent example of the genetic algorithms to which our analysis applies---the result of Section \ref{merger_ordering} has the at-first surprising implication that, in particular edge cases, multinomial resampling can yield fewer coalescences to a common ancestor than a minimum-variance scheme such as stratified or systematic resampling.

We consider an interacting particle system in which a population of $N$ particles evolves in discrete generations.
The first generation at time zero is initialised at respective locations $X_0 := ( X_0( 1 ), \ldots, X_0( N ) )$, where $X_k(i) \in \mathcal{X}$ and we take $\mathcal{X}$ to be an arbitrary Polish space, typically $\calX \subseteq \bbR^d$.
Subsequent generations are driven by a family of non-negative potential functions $g_k : \calX \to (0, \infty)$ and Markov kernels $M_k : \calX \to \calP( \calX )$, where $\calP( \calX )$ is the set of probability measures on $\calX$.
We will assume that each $M_k(x,\cdot)$ admits a density with respect to Lebesgue measure for all $x$ and will use the same symbol for the kernel and its density.

Let $a_k = ( a_k( 1 ), \ldots, a_k( N ) )$ be a random tuple of ancestor indices taking values in $[ N ] := \{ 1, \ldots, N \}$, with distinct generations $\{ a_k | X_k, X_{ k + 1 } \}_{ k \geq 0 }$ being conditionally independent given particle locations.
We require that
\begin{equation*}
\bbP( a_k( i ) = m | X_k ) = \frac{ g_k( X_k( m ) ) }{ \sum_{ v = 1 }^N g_k( X_k( v ) ) }
\end{equation*}
marginally for each $k$.
The ill-defined event $g_k(X_k(1)) = \ldots = g_k(X_k(N)) = 0$ is ruled out by the assumption that $g_k(x) > 0$ for all $x \in \calX$.
The joint distribution of the entries of $a_k$ is determined by the so-called resampling mechanism (see e.g.\ \cite[Chapter 9]{chopin2020}).
Particle locations at generation $k + 1$ given those in generation $k$ are obtained by sampling an ancestor vector $a_k$ from its conditional distribution given $X_k$, whereupon
\begin{equation*}
\bbP\Bigg( X_{ k + 1 } \in \rmd x \Big| X_k, a_k \Bigg) = \prod_{ i = 1 }^N M_k( X_k( a_k( i ) ), x( i ) ) \rmd x( i ).
\end{equation*}
Throughout, we work on a probability space $( \Omega, \mathcal{F}, \bbP )$ which is rich enough to support a sequence of such particle systems for $N \in \bbN$.

For fixed $N$, this particle system can be seen as a non-neutral Wright--Fisher model of evolution by interpreting particle locations as genetic alleles, the potentials $g_k$ as fitnesses, and the Markov kernels $M_k$ as mutation processes \cite[Section 2]{delmoral2009}.
They also describe SMC methods: a very broad class of algorithms used in computational statistics and related disciplines \cite{chopin2020, delmoral2004}.
In both settings, the genealogical tree embedded into the particle system by the ancestor vectors is known to be important \cite{delmoral2016, delmoral2001b, delmoral2009, jacob2015, kingman1982coal, kingman1982gene, lee2018, mohle1998, mohle1999, mohle2001}.
To describe these genealogies, it will be convenient to assume that the particle system has been run for a large number of generations, and to relabel time in reverse: the terminal generation will be generation zero, their parents are generation one, etc.
We will adopt this relabelling for the remainder of the manuscript with a few exceptions for which the direction of time is stated explicitly.
Our scaling limit will require the particle system to be well-defined for all $k \in \bbN$ in this reverse-time labelling.
However, we do not require the particle system to be stationary.

For a countable set $A$, define $A_d^n \subset A^n$ as the subset of $n$-tuples of distinct elements.
Similarly, let $A_{ d, u }^n$ denote the set of unordered size-$n$ subsets of $A$.
Let $| v |$ denote the number of elements in an arbitrary tuple $v$, and let $\bbN_0 := \bbN \cup \{ 0 \}$.
It is convenient to define the genealogical process of $n$ generation-zero particles as a stochastic process $\{ G_k^{ N, n } \}_{ k \geq 0 }$ taking values in labelled partitions of $[n]$, with the $i$th partition block labelled by the index of the corresponding particle.
Throughout, we think of partition blocks ordered lexicographically for concreteness, but none of our results rely on that ordering.
We set $G_0^{ N, n } = \{ ( \{ 1 \}, \ell^1 ), \ldots, ( \{ n \}, \ell^n ) \}$ for indices $( \ell^1, \ldots, \ell^n ) \in [ N ]_d^n$.
Two or more blocks merge when the corresponding particles share an ancestor, and the resulting block is labelled by the particle index of that ancestor.
Because $\{ G_k^{ N, n } \}_{ k \geq 0 }$ does not track particle locations $\{ X_k \}_{ k \geq 0 }$ and hence location-dependent fitnesses $\{ g_k \}_{ k \geq 0 }$, it is not a Markov process in general.
We will denote by $\bar{G}_k^{ N, n }$ the partition-valued process obtained from $G_k^{ N, n }$ by removing partition labels, and write $\xi \prec \eta$ when partition $\eta$ is obtained from partition $\xi$ by merging exactly two blocks.
The setup and notation are illustrated in Figure \ref{wright-fisher}.
\begin{figure}[!ht]
\centering
\scriptsize{
	\begin{tikzpicture}
		\node[circle,draw] (X10) {};
		\node[circle,draw] (X20) [below=of X10] {};
		\node[circle,draw] (X30) [below=of X20] {};
		\node[circle,draw] (X40) [below=of X30] {};
		\node[circle,draw] (XN0) [below=of X40] {};
		\node[circle,draw] (X11) [right=of X10] {};
		\node[circle,draw] (X21) [below=of X11] {};
		\node[circle,draw] (X31) [below=of X21] {};
		\node[circle,draw] (X41) [below=of X31] {};
		\node[circle,draw] (XN1) [below=of X41] {};
		\node[circle,draw] (X12) [right=of X11] {};
		\node[circle,draw] (X22) [below=of X12] {};
		\node[circle,draw] (X32) [below=of X22] {};
		\node[circle,draw] (X42) [below=of X32] {};
		\node[circle,draw] (XN2) [below=of X42] {};
		\node[circle,draw] (X13) [right=of X12] {};
		\node[circle,draw] (X23) [below=of X13] {};
		\node[circle,draw] (X33) [below=of X23] {};
		\node[circle,draw] (X43) [below=of X33] {};
		\node[circle,draw] (XN3) [below=of X43] {};
		\node[circle,draw] (X14) [right=of X13] {};
		\node[circle,draw] (X24) [below=of X14] {};
		\node[circle,draw] (X34) [below=of X24] {};
		\node[circle,draw] (X44) [below=of X34] {};
		\node[circle,draw] (XN4) [below=of X44] {};
		\node (L1) [right=of X14] {$G_4^{(5, 5)} = \{ ( \{ 1, 2, 3, 4, 5 \}, 4 ) \}$};
		\node (L2) [right=of X24] {$G_3^{(5, 5)} = \{ ( \{ 1, 2, 3 \}, 2 ), ( \{ 4, 5 \}, 3 ) \}$};
		\node (L3) [right=of X34] {$G_2^{(5, 5)} = \{ ( \{ 1, 2, 3 \}, 2 ), ( \{ 4, 5 \}, 4 ) \}$};
		\node (L4) [right=of X44] {$G_1^{(5, 5)} = \{ ( \{ 1 \}, 1 ), ( \{ 2 \}, 2 ), ( \{ 3 \}, 3 ), ( \{ 4, 5 \}, 4 ) \}$};
		\node (LN) [right=of XN4] {$G_0^{(5, 5)} = \{ ( \{ 1 \}, 1 ), ( \{ 2 \}, 2 ), ( \{ 3 \}, 3 ), ( \{ 4 \}, 4 ), ( \{ 5 \}, 5) \}$};

		\path (X40) edge[-latex, line width=2] (XN0);
		\path (X41) edge[-latex, line width=2] (XN1);
		\path (X42) edge[-latex, line width=2] (XN2);
		\path (X43) edge[-latex, line width=2] (XN3);
		\path (X43) edge[-latex, line width=2] (XN4);
		
		\path (X31) edge[-latex, line width=2] (X40);
		\path (X31) edge[-latex, line width=2] (X41);
		\path (X31) edge[-latex, line width=2] (X42);
		\path (X33) edge[-latex, line width=2] (X43);
		\path (X34) edge[-latex] (X44);
		
		\path (X21) edge[-latex] (X30);
		\path (X21) edge[-latex, line width=2] (X31);
		\path (X22) edge[-latex] (X32);
		\path (X22) edge[-latex, line width=2] (X33);
		\path (X24) edge[-latex] (X34);

		\path (X11) edge[-latex] (X20);
		\path (X13) edge[-latex, line width=2] (X21);
		\path (X13) edge[-latex, line width=2] (X22);
		\path (X13) edge[-latex] (X23);
		\path (X14) edge[-latex] (X24);
	\end{tikzpicture}
	}
\caption{An example realisation of the interacting particle model along with the corresponding realisation of the genealogical process.
Each row is a generation consisting of $N=5$ particles with labels $1, \ldots, 5$.
The arrows point in the direction of the time-evolution of the particle system, while the time-index of the genealogical process counts generations in reverse.
Edges highlighted in bold form the ancestral tree of the population in generation 0.}
\label{wright-fisher}
\end{figure}

In \cite{koskela2020annals} we attempted to prove that, when time is suitably rescaled, the unlabelled genealogical process $\{ \bar{ G }_k^{ N, n } \}_{k \geq 0}$ converges to the Kingman coalescent as $N \to \infty$.
Our proof relied on the following formula, which is incorrect in general, for conditional transition probabilities between partitions $\xi$ and $\eta$, when $\eta$ is obtained from $\xi$ by merging some subsets of blocks:
\begin{equation}\label{eq:neutral_transition_probability}
\bbP( \bar{ G }_k^{ N, n } = \eta | \bar{ G }_{ k - 1 }^{ N, n } = \xi, \nu_k ) = \frac{ 1 }{ ( N )_{ | \xi | } } \sum_{ ( i_1, \ldots, i_{ | \eta | } ) \in [ N ]_d^{ | \eta | } } ( \nu_k( i_1 ) )_{ b_1 } \ldots ( \nu_k( i_{ | \eta | } ) )_{ b_{ | \eta | } },
\end{equation}
where $( x )_j = x ( x - 1 ) \ldots ( x - j + 1)$ is the falling factorial, $b_i$ is the number of blocks in $\xi$ which were merged to obtain the $i$th block of $\eta$, and $\nu_k := ( \nu_k( 1 ), \ldots, \nu_k( N ) )$ are the family sizes in generation $k$:
\begin{equation*}
\nu_k( i ) := | \{ j \in [ N ] : a_k( j ) = i \} |.
\end{equation*}
The assumption underlying \eqref{eq:neutral_transition_probability} has been called the \textit{random assignment condition}, since the right-hand side arises as the correct transition probability whenever conditional assignment of offspring to parents given family sizes is uniform (c.f.\ \cite[page 439, assumption 2]{mohle1998}).
The counterexample in \cite{rubenthaler:2023+} shows that this formula is valid when fitness is not hereditary so that $\nu_k$ and $\nu_j$ are independent whenever $k \neq j$, but is not a correct description of the underlying particle system when fitness is inherited because family sizes do not D-separate generations \cite[Figure 1]{brownetal:2021}.
See \cite[Chapter 3]{pearl:1988} for details of D-separation.
In settings with hereditary fitness, the probability which would be required in the convergence argument of \cite{koskela2020annals} instead of \eqref{eq:neutral_transition_probability} is 
\begin{equation*}
\bbP( G_k^{ N, n } = \eta | G_{ k - 1 }^{ N, n } = \xi, \nu_1, \nu_2, \nu_3, \ldots ),
\end{equation*}
which is not equal to the right-hand side of \eqref{eq:neutral_transition_probability}.
Our Theorem \ref{thm:main_result} replaces conditioning on family sizes with conditioning on particle locations (or equivalently, particle fitnesses), which suffice for D-separation.
Hence our theorem applies to models with hereditary fitness, for which random assignment does not hold.

In this article we replace \eqref{eq:neutral_transition_probability} with a valid expression, and show that the main results of \cite{brown:2021, brownetal:2021, brownetal:2023, koskela2020annals} are true under an additional assumption (specifically, \eqref{kingman_assumption_4} in Theorem \ref{thm:main_result} below).
In addition to the correction, we also present a substantially simplified proof which is shorter than its equivalent in either \cite{koskela2020annals} or \cite{brownetal:2023}, despite the fact that the theorem statement combines the main results of both of these predecessor articles.
The proof technique, which may be of independent interest, is based on analysing single holding times of a non-Markovian jump process, and stitching them together to obtain a Markovian scaling limit.

The precise assumptions leading to our main result, Theorem \ref{thm:main_result}, are rather technical.
However, they can be expected to hold for populations whose fitnesses $g_k$ are bounded above and away from zero, and where the mixing of fitnesses by the mutation kernels $M_k$ is faster than the rate with which lineages coalesce to common ancestors.
The assumption of fast mutation is reminiscent of a similar result by \cite{hossjer:2011} on genealogies in spatially structured populations under rapid spatial motion.
It is also a very strong assumption from the biological point of view where mutation is typically a slow process, ruling out, e.g., the ancestral selection graph \cite{krone1997}.

In the SMC context, we verify that the conditions of Theorem \ref{thm:main_result} hold under strong but standard mixing assumptions on the potentials $g_k$ and mutation kernels $M_k$ (see \eqref{eq:strong_mixing} in Proposition \ref{prop:multinomial}, as well as \eqref{eq:m_mixing} and \eqref{eq:uniform_weight_bound} in Proposition \ref{prop:stratified}).
These conditions essentially rule out non-compact state spaces but are widespread in the SMC literature, and yield much stronger ergodicity than our results require \cite{karjalainenetal:2023+}.
Hence, it is likely that they could be relaxed, at least in particular cases.
Many results are known to be robust to violations of strong mixing in practice \cite{cerou2011, chopin2004, delmoral2001, jacob2015, kuensch2005}, among them the numerical simulations in \cite[Section 3]{koskela2020annals} which did not satisfy the requisite assumptions, but for which the predictions of our scaling limit were accurate.
For these reasons, and as evidenced by the material in Section \ref{section:resampling_schemes}, we regard SMC algorithms as the primary motivation and domain of application for this work.

At first glance, the fast mutation regime may seem an uninteresting one.
In  biological contexts one is typically interested in non-neutral processes in which selection is significant, and consequently has a material effect on the genealogy, even asymptotically.
In contrast, in SMC and similar algorithms selection is something of a mechanism of last resort: it allows us to correct for an inability to sample from the true distribution of interest but comes at a price.
Many of the innovations in SMC over the past three decades can be viewed as methods to mitigate the impact of selection, and to avoid it, materially altering the resulting genealogical structure.
As such, in the computational domain, developing a good understanding of settings in which selection is not a dominant effect is of substantial interest.

\section{A restated convergence theorem}

Let $\bfX := \{ X_k \}_{ k \geq 0 }$ be the locations of all particles in all generations in a particle system run for an infinite number of generations.
For $N \geq 2$ and $2 \leq n \leq N$, let $\xi$ be a partition of $[n]$ with at least two blocks, and let $\ell = ( \ell^1, \ldots, \ell^{ | \xi | } )$ be a labelling of the partition blocks with distinct elements of $[N]$.
With the convention that $\sum_{k = j + 1}^j = 0$, let 
\begin{align}
c_N( \xi, \ell, j; k ) &:= \binom{ | \xi | }{ 2 }^{ -1 } \bbP^{ \bfX }( | \bar{ G }_k^{N, n} | < | \xi | | \bar{ G }_{ k - 1 }^{N, n} = \xi, G_j^{N, n} = ( \xi, \ell ) ), \label{kingman_timescale} \\
\tau_N( \xi, \ell, j; t ) &:= \min\Bigg\{ s \geq j : \sum_{ k = j + 1 }^s c_N( \xi, \ell, j; k ) \geq t \Bigg\} \notag 
\end{align}
be, respectively, the scaled conditional probability of at least one merger in generation $k$ given particle locations $\bfX$, particle labels $\ell$ in generation $j$, and no mergers in the intervening generations, and a family of its generalised inverses for each starting generation $j$, where $t \geq 0$.
The latter will turn out to be the appropriate timescale for obtaining our scaling limit.
The scaling of \eqref{kingman_timescale} by $1 / \binom{|\xi|}{2}$ is cosmetic and chosen to ensure that the Kingman coalescent arises as our scaling limit.

In order to state our main result, let $n \in \bbN$, $\xi = \{ \{ 1 \}, \ldots, \{ n \} \}$, and $\ell \in [ N ]_d^n$ be fixed but arbitrary.
We define $T_n$ as the first jump time of 
\begin{equation*}
( \bar{ G }_{ \tau_N( \xi, \ell, 0; t ) }^{ N, n } )_{ t \geq 0 },
\end{equation*}
and denote the end point of that jump by $( \xi( T_n ), \ell( T_n ) )$.
We also abbreviate $\tau_{ N, n }( t ) \equiv \tau_N( \xi, \ell, 0; t )$.
For $i \in \{ n - 1 , \ldots, 2 \}$, we iteratively define 
\begin{equation*}
\tau_{ N, i }( t ) \equiv \tau_N( \xi( T_{ i + 1 } ), \ell( T_{ i + 1 } ), \tau_{ N, i + 1 }( T_{ i + 1 } ); t ),
\end{equation*}
where $T_i$ is the time of the first jump of $( \bar{ G }_{ \tau_{ N, i }( t ) } )_{ t \geq 0 }$, and $( \xi( T_i ), \ell( T_i ) )$ denotes the end point of that jump.
To lighten notation we let $T_{ n + 1 } \equiv 0$, define $S_i := T_n + \ldots + T_i$ for $i = 2, \ldots, n$, and introduce the short-hand
\begin{equation*}
\bar{ G }_{ \tau_N( t ) }^{ N, n } := \sum_{ i = 2 }^n \bar{G}_{ \tau_{ N, i }( t - S_{ i + 1 } ) }^{ N, i } \mathds{ 1 }_{ [ S_{ i + 1 }, S_i ) }( t ), 
\end{equation*}
where the right-hand side is the concatenation of the genealogical processes for population size $N$ while there are $| \xi( T_{ i + 1 } ) |$ lineages.
This also implicitly defines shorthand for the concatenated timescale
\begin{equation*}
\tau_N( t ) := \sum_{ i = 2 }^n \tau_{ N, i }( t - S_{ i + 1 } ) \mathds{ 1 }_{ [ S_{ i + 1 }, S_i ) }( t ).
\end{equation*}
Note that the $t$-argument of each $\tau_{N, i}(t)$ is local, but that the ranges of these time changes are the global generations of the underlying particle system, so that e.g.\ $\tau_N( \xi, \ell, j; 0) = j$.
The concatenated timescale $\tau_N(t)$ joins the local $t$-variables of $\{ \tau_{N, i}( \cdot ) \}_{i = 2}^n$ into one global, non-decreasing timescale.

We will also let $( \Pi_t^n )_{ t \geq 0 }$ denote the Kingman $n$-coalescent, that is, the Markov process taking values in partitions of $[n]$ with $\Pi_0^n = \{ \{ 1 \}, \ldots, \{ n \} \}$ and in which each pair of blocks merges at unit rate, so that the rate with which two random blocks merge is $\binom{k}{2}$ while $| \Pi_t^n | = k$.

\begin{thm}\label{thm:main_result}
Suppose that for any $n \in \bbN$, $j < k \in \bbN_0$, $t \in ( 0, \infty )$, any partition $\xi$ of $[n]$ and any $\eta$ such that $\xi \prec \eta$, and any $\ell \in [ N ]_d^{ | \xi | }$,
\begin{align}
\lim_{ N \to \infty } \bbE[ c_N( \xi, \ell, j; k ) | X_j ] &= 0, \label{kingman_assumption_1} \\
\lim_{ N \to \infty } \bbE[ c_N( \xi, \ell, j; \tau_N( \xi, \ell, j; t ) ) | X_j ] &= 0, \label{kingman_assumption_2} \\
\lim_{ N \to \infty } \bbE\Bigg[ \sum_{ k = j + 1 }^{ \tau_N( \xi, \ell, j; t ) } c_N( \xi, \ell, j; k )^2 \Big| X_j \Bigg] &= 0, \label{kingman_assumption_3} \\
\lim_{ N \to \infty } \frac{ \bbP^{ \bfX }( \bar{ G }_{ \tau_N( \xi, \ell, j; t ) }^{ N, n } = \eta | \bar{ G }_{ \tau_N( \xi, \ell, j; t ) - 1 }^{ N, n } = \xi, G_j^{ N, n } = ( \xi, \ell ) ) }{ c_N( \xi, \ell, j; \tau_N( \xi, \ell, j; t ) ) } &= 1, \label{kingman_assumption_4}\\
\lim_{N \to \infty} \bbP( \tau_N( \xi, \ell, j; t ) = \infty ) &= 0, \label{timescale_assumption}
\end{align}
hold $\bbP$-almost surely.
Then, for any fixed $n \in \bbN$,
\begin{equation*}
\lim_{ N \to \infty } ( \bar{ G }_{ \tau_N( t ) }^{ N, n } )_{ t \geq 0 } = ( \Pi_t^n )_{ t \geq 0 }
\end{equation*}
weakly in the Skorokhod $J_1$ topology on the space of right-continuous paths with left limits.
\end{thm}

\begin{rem}
Assumption \eqref{kingman_assumption_4} is strong.
Heuristically, it is satisfied by particle systems in which particle locations mix fast enough that any information about particle weights at the time of a merger is lost by the time of the following merger.
In such a particle system, the fact that the parent of two (or more) coalescing particles is likely to be fitter than average will not inform particle fitnesses at the time of the next coalescence event.
In Section \ref{section:resampling_schemes} we verify that standard resampling schemes satisfy \eqref{kingman_assumption_4} under standard strong mixing conditions, and also make explicit links between the assumptions \eqref{kingman_assumption_1}--\eqref{kingman_assumption_4} and the model ingredients $g_k$ and $M_k$.
The exact assumptions needed on $g_k$, $M_k$ to ensure \eqref{kingman_assumption_1}--\eqref{kingman_assumption_4} depend on the resampling scheme.
The fact that we only require \eqref{kingman_assumption_4} to hold asymptotically is crucial; the assumption cannot be expected to hold for finite $N$ even for neutral particle systems because more than two lineages can merge in one generation with positive probability.
Assumption \eqref{timescale_assumption} ensures that the limiting Kingman coalescent has an infinite lifetime.
\end{rem}

\begin{rem}
In the neutral case, when the potentials $g_k$ are constant functions and \eqref{kingman_assumption_4} holds by construction, a condition analogous to
\begin{equation*}
\lim_{ N \to \infty }\frac{ \bbE[ \bbP^{ \bfX }( | \bar{ G }_k^{ N, n } | = 1 \big| | \bar{ G }_{ k - 1 }^{ N, n } | = 3 ) ] }{ \bbE[ \bbP^{ \bfX }( | \bar{ G }^{ N, n }_k | = 1 \big| | \bar{ G }_{ k - 1 }^{ N, n } | = 2 ) ] } = 0
\end{equation*}
is necessary and sufficient for weak convergence of suitably time-rescaled genealogical processes to the Kingman coalescent \cite{mohle2003}.
In particular, it implies \eqref{kingman_assumption_1}--\eqref{kingman_assumption_3} in that setting.
However, these implications rely on an explicit transition probability formula resembling \eqref{eq:neutral_transition_probability}, which does not hold in the general non-neutral case.
Hence, we need to resort to the more cumbersome conditions \eqref{kingman_assumption_1}--\eqref{kingman_assumption_4}.
\end{rem}

\begin{proof}
We will prove the theorem in four parts: first by showing convergence of the holding time until a jump of $( \bar{G}_{ \tau_N( \xi, \ell, j; t ) }^{ N, n } )_{ t \geq 0 }$ to an exponentially distributed random variable with rate $\binom{n}{2}$ in Part 1, second by showing that the merger event is between exactly two uniformly chosen blocks in Part 2, and third by demonstrating that we can concatenate waiting times and mergers to construct the whole process from initial condition $\{ \{ 1 \}, \ldots, \{ n \} \}$ to the most recent common ancestor $\{ \{ 1, \ldots, n \} \}$ in Part 3.
Conditioning on $X_j$ in \eqref{kingman_assumption_1}--\eqref{kingman_assumption_3} will appear superfluous until the third step.
In Part 4, we control the modulus of continuity to prove weak convergence.

\textbf{Part 1.} We begin by showing convergence of the holding time until a jump of $( \bar{G}_{ \tau_N( \xi, \ell, j; t ) }^{ N, n } )_{ t \geq 0 }$ to an exponentially distributed random variable with rate $\binom{n}{2}$.

From an initial labelled partition $( \xi, \ell )$ in generation $j \in \bbN_0$, the conditional survivor function of the next jump given particle locations $\bfX$ is
\begin{align}
&\bbP^{ \bfX }(  \bar{ G }_{ \tau_N( \xi, \ell, j; t ) }^{ N, n } = \xi | G_j^{ N, n } = ( \xi, \ell ) ) \notag \\
&= \bbP^{ \bfX }( \bar{ G }_{ j + 1 }^{ N, n } = \xi, \ldots, \bar{ G }_{ \tau_N( \xi, \ell, j; t ) }^{ N, n } = \xi | G_j^{ N, n } = ( \xi, \ell ) ) \notag \\
&= \prod_{ k = j + 1 }^{ \tau_N( \xi, \ell, j; t ) } \bbP^{ \bfX }( \bar{ G }_k^{ N, n } = \xi | \bar{ G }_{ k - 1 }^{ N, n } = \xi, G_j^{ N, n } = ( \xi, \ell ) ) \notag \\
&= \prod_{ k = j + 1 }^{ \tau_N( \xi, \ell, j; t ) } \left[ 1 - \bbP^{ \bfX }\left( | \bar{ G }_k^{ N, n } | < | \xi | \middle| \bar{ G }_{ k - 1 }^{ N, n } = \xi, G_j^{ N, n } = ( \xi, \ell ) \right) \right], \label{survivor_function_ub}
\end{align}
because $\tau_N(\xi, \ell, j; t)$ is $\sigma(\bfX)$-measurable.
We now need the generic inequality
\begin{align}
&( - 1 )^{ \alpha } \sum_{ ( k_1, \ldots, k_{ \alpha } ) \in [ \tau ]_d^{ \alpha } } \prod_{ m = 1 }^{ \alpha } c( k_m ) \notag\\
&\leq ( - 1 )^{ \alpha }\sum_{(k_1, \ldots, k_{ \alpha }) \in[\tau]^\alpha}  \prod_{ m = 1 }^{ \alpha } c( k_m ) + \binom{ \alpha }{ 2 } \sum_{ k = 1 }^{ \tau } c( k )^2 \sum_{(k_1, \ldots, k_{ \alpha - 2 }) \in [\tau]^{\alpha-2} } \prod_{ m = 1 }^{ \alpha - 2 } c( k_m ), \label{eq:repeated_terms_bound}
\end{align}
for $\tau \geq \alpha \geq 1$ and coefficients $c( k ) \geq 0$, which follows from Lemma \ref{lem:multinomial} in Appendix~\ref{app:lemmas} by noting that
\begin{align*}
\Bigg(\sum_{i=1}^N x_i\Bigg)^\alpha &= \sum_{i_1,\dots,i_\alpha = 1}^N \prod_{m=1}^\alpha x_{i_m},
\end{align*}
and rearranging.
In particular if $\alpha$ is odd then multiply \eqref{eq:multinomial} in the Appendix by $(-1)^\alpha$ (which reverses the inequality) and rearrange; if $\alpha$ is even then \eqref{eq:repeated_terms_bound} is trivial because each term on the left-hand side can be matched with a term in the first sum on the right-hand side.

Expanding the product on the right-hand side of \eqref{survivor_function_ub} and using \eqref{eq:repeated_terms_bound}, we obtain
\begin{align*}
&\bbP^{ \bfX }(  \bar{ G }_{ \tau_N( \xi, \ell, j; t ) }^{ N, n } = \xi | G_j^{ N, n } = ( \xi, \ell ) )  \\
&= 1 + \sum_{ \alpha = 1 }^{ \tau_N( \xi, \ell, j; t ) - j } \sum_{ j + 1 \leq k_1 < \ldots < k_{ \alpha } \leq \tau_N( \xi, \ell, j; t ) } ( -1 )^{ \alpha } \\
&\phantom{= 1 + \sum_{ \alpha = 1 }^{ \tau_N( \xi, \ell, j; t ) - j } \sum_{ j + 1 \leq k_1 < \ldots < k_{ \alpha }}} \times  \prod_{ m = 1 }^{ \alpha } \bbP^{ \bfX }( | \bar{ G }_{ k_m }^{ N, n } | < | \xi | | \bar{ G }_{ k_m - 1 }^{ N, n } = \xi, G_j^{ N, n } = ( \xi, \ell ) ) \\
&= 1 + \sum_{ \alpha = 1 }^{ \tau_N( \xi, \ell, j; t ) - j } \frac{ ( -1 )^{ \alpha } }{ \alpha ! } \binom{ | \xi | }{ 2 }^{ \alpha } \sum_{ ( k_1, \ldots, k_{ \alpha } ) \in \{ j + 1, \ldots, \tau_N( \xi, \ell, j; t ) \}_d^{ \alpha } } \prod_{ m = 1 }^{ \alpha } c_N( \xi, \ell, j; k_m ) \\
&\leq 1 + \sum_{ \alpha = 1 }^{ \tau_N( \xi, \ell, j; t ) - j } \frac{ ( -1 )^{ \alpha } }{ \alpha ! } \binom{ | \xi | }{ 2 }^{ \alpha } \sum_{ k_1, \ldots, k_{ \alpha } = j + 1 }^{ \tau_N( \xi, \ell, j; t ) } \prod_{ m = 1 }^{ \alpha } c_N( \xi, \ell, j; k_j ) \\
&\phantom{\leq 1} + \sum_{ \alpha = 2 }^{ \tau_N( \xi, \ell, j; t ) - j } \frac{ 1 }{ \alpha ! } \binom{ | \xi | }{ 2 }^{ \alpha } \binom{ \alpha }{ 2 } \sum_{ k = j + 1 }^{ \tau_N( \xi, \ell, j; t ) } c_N( \xi, \ell, j; k )^2  \sum_{ k_1, \ldots, k_{ \alpha - 2 } = j + 1 }^{ \tau_N( \xi, \ell, j; t ) } \prod_{ m = 1 }^{ \alpha - 2 } c_N( \xi, \ell, j; k_m ) \\
&= \sum_{ \alpha = 0 }^{ \tau_N( \xi, \ell, j; t ) - j } \frac{ ( -1 )^{ \alpha } }{ \alpha ! } \binom{ | \xi | }{ 2 }^{ \alpha } \Bigg( \sum_{ k = j + 1 }^{ \tau_N( \xi, \ell, j; t ) } c_N( \xi, \ell, j; k ) \Bigg)^{ \alpha } \\
&\phantom{\leq 1}+ \sum_{ \alpha = 2 }^{ \tau_N( \xi, \ell, j; t ) - j } \frac{ 1 }{ \alpha ! } \binom{ | \xi | }{ 2 }^{ \alpha } \binom{ \alpha }{ 2 } \Bigg( \sum_{ k = j + 1 }^{ \tau_N( \xi, \ell, j; t ) } c_N( \xi, \ell, j; k ) \Bigg)^{ \alpha - 2 } \sum_{ k = j + 1 }^{ \tau_N( \xi, \ell, j; t ) } c_N( \xi, \ell, j; k )^2.
\end{align*}
By definition of $\tau_N( \xi, \ell, j; t )$,
\begin{equation}\label{generalised_inverse_sandwich}
t \leq \sum_{ k = j + 1 }^{ \tau_N( \xi, \ell, j; t ) } c_N( \xi, \ell, j; k ) \leq t + c_N( \xi, \ell, j; \tau_N( \xi, \ell, j; t ) ) \leq t + 1,
\end{equation}
which yields
\begin{align*}
&\bbP^{ \bfX }(  \bar{ G }_{ \tau_N( \xi, \ell, j; t ) }^{ N, n } = \xi | G_j^{ N, n } = ( \xi, \ell ) ) \\
&\leq \sum_{ \alpha = 0, \text{ even} }^{ \tau_N( \xi, \ell, j; t ) - j } \frac{ ( -1 )^{ \alpha } }{ \alpha ! } \binom{ | \xi | }{ 2 }^{ \alpha } [ t + c_N( \xi, \ell, j; \tau_N( \xi, \ell, j; t ) ) ]^{ \alpha } + \sum_{ \alpha = 1, \text{ odd} }^{ \tau_N( \xi, \ell, j; t ) - j } \frac{ ( -t )^{ \alpha } }{ \alpha ! } \binom{ | \xi | }{ 2 }^{ \alpha } \\
&\phantom{\leq} + \frac{ 1 }{ 2 } \binom{ | \xi | }{ 2 }^2 \sum_{ \alpha = 2 }^{ \tau_N( \xi, \ell, j; t ) - j } \frac{ 1 }{ ( \alpha - 2 ) ! } \binom{ | \xi | }{ 2 }^{ \alpha - 2 } ( t + 1 )^{ \alpha - 2 } \sum_{ k = j + 1 }^{ \tau_N( \xi, \ell, j; t ) } c_N( \xi, \ell, j; k )^2
\end{align*}
For $c \in [0, 1]$ we also have
\begin{equation}\label{elementary_binomial_bound}
( t + c )^{ \alpha } = \sum_{ \beta = 0 }^{ \alpha } \binom{ \alpha }{ \beta } t^{ \beta } c^{ \alpha - \beta } = t^{ \alpha } + c \sum_{ \beta = 0 }^{ \alpha - 1 } \binom{ \alpha }{ \beta } t^{ \beta } c^{ \alpha - 1 - \beta } \leq t^{ \alpha } + \alpha ( t + 1 )^{ \alpha - 1 } c,
\end{equation}
so that
\begin{align*}
&\bbP^{ \bfX }(  \bar{ G }_{ \tau_N( \xi, \ell, j; t ) }^{ N, n } = \xi | G_j^{ N, n } = ( \xi, \ell ) ) \\
&\leq \sum_{ \alpha = 0, \text{ even} }^{ \tau_N( \xi, \ell, j; t ) - j } \frac{ ( -1 )^{ \alpha } }{ \alpha ! } \binom{ | \xi | }{ 2 }^{ \alpha } [ t^{ \alpha } + \alpha ( t + 1 )^{ \alpha - 1 } c_N( \xi, \ell, j; \tau_N( \xi, \ell, j; t ) ) ] \\
&\phantom{\leq} + \sum_{ \alpha = 1, \text{ odd} }^{ \tau_N( \xi, \ell, j; t ) - j } \frac{ ( -1 )^{ \alpha } }{ \alpha ! } \binom{ | \xi | }{ 2 }^{ \alpha } t^{ \alpha } \\
&\phantom{\leq} + \frac{ 1 }{ 2 } \binom{ | \xi | }{ 2 }^2 \sum_{ \alpha = 0 }^{ \tau_N( \xi, \ell, j; t ) - j } \frac{ 1 }{ \alpha ! } \binom{ | \xi | }{ 2 }^{ \alpha } ( t + 1 )^{ \alpha } \sum_{ k = j + 1 }^{ \tau_N( \xi, \ell, j; t ) } c_N( \xi, \ell, j; k )^2 \\
&\leq \sum_{ \alpha = 0 }^{ \infty } \binom{ | \xi | }{ 2 }^{ \alpha } \frac{ ( -t )^{ \alpha } }{ \alpha ! } \mathds{ 1 }\{ \tau_N( \xi, \ell, j; t ) \geq \alpha + j \} \\
&\phantom{\leq} +  c_N( \xi, \ell, j; \tau_N( \xi, \ell, j; t ) ) \binom{ | \xi | }{ 2 } \sum_{ \alpha = 1 }^{ \infty } \frac{ ( t + 1 )^{ \alpha - 1 } }{ ( \alpha - 1 ) ! } \binom{ | \xi | }{ 2 }^{ \alpha - 1 } \\
&\phantom{\leq} + \frac{ 1 }{ 2 } \binom{ | \xi | }{ 2 }^2 \sum_{ \alpha = 0 }^{ \tau_N( \xi, \ell, j; t ) - j } \frac{ 1 }{ \alpha ! } \binom{ | \xi | }{ 2 }^{ \alpha } ( t + 1 )^{ \alpha } \sum_{ k = j + 1 }^{ \tau_N( \xi, \ell, j; t ) } c_N( \xi, \ell, j; k )^2 \\
&\leq \sum_{ \alpha = 0 }^{ \infty } \frac{ ( -t )^{ \alpha } }{ \alpha ! } \binom{ | \xi | }{ 2 }^{ \alpha } \mathds{ 1 }\{ \tau_N( \xi, \ell, j; t ) \geq \alpha + j \} \\
&\phantom{\leq} + \exp\Bigg( \binom{ | \xi | }{ 2 } ( t + 1 ) \Bigg)  \Bigg(  \binom{ | \xi | }{ 2 } c_N( \xi, \ell, j; \tau_N( \xi, \ell, j; t ) ) + \frac{ 1 }{ 2 }  \binom{ | \xi | }{ 2 }^2 \sum_{ k = j + 1 }^{ \tau_N( \xi, \ell, j; t ) } c_N( \xi, \ell, j; k )^2 \Bigg).
\end{align*}
To show that $\lim_{ N \to \infty } \bbE[ \mathds{ 1 }\{ \tau_N( \xi, \ell, j; t ) > \alpha + j \} | X_j ] = 1$ we follow the argument on \cite[Page 572]{koskela2020annals}: by definition of $\tau_N$, Markov's inequality, and \eqref{kingman_assumption_1},
\begin{align}
\lim_{ N \to \infty } \bbE[ \mathds{ 1 }\{ \tau_N( \xi, \ell, j; t ) > \alpha \} | X_j ]&= 1 - \lim_{ N \to \infty } \bbP( \tau_N( \xi, \ell, j; t ) \leq \alpha | X_j ) \notag \\
&= 1 - \lim_{ N \to \infty } \bbP\Bigg( \sum_{ k = j + 1 }^{ \alpha } c_N( \xi, \ell, j; k ) \geq t \Big| X_j \Bigg) \notag \\
&\geq 1 - \lim_{ N \to \infty } \sum_{ k = 1 }^{ \alpha } \frac{ \bbE[ c_N( \xi, \ell, j; k ) | X_j ] }{ t } = 1 \label{indicator_limit_2}
\end{align}
for every $\alpha \in \bbN$.
Hence, by \eqref{kingman_assumption_2}, \eqref{kingman_assumption_3}, and \eqref{indicator_limit_2},
\begin{equation} \label{exponential_upper}
\lim_{ N \to \infty } \bbE[ \bbP^{ \bfX }( \bar{ G }_{ \tau_N( \xi, \ell, j; t ) }^{ N, n } = \xi | G_j^{ N, n } = ( \xi, \ell ) ) | X_j ] \leq \exp\Bigg( - \binom{ | \xi | }{ 2 } t \Bigg),
\end{equation}
which is the survivor function of the holding time of the Kingman coalescent started from partition $\xi$.

For a corresponding lower bound, we will need the following variant of \eqref{eq:repeated_terms_bound} which follows from \eqref{eq:multinomial} in Appendix~\ref{app:lemmas} if $\alpha$ is even, and which is trivial when $\alpha$ is odd:
\begin{align}
&( - 1 )^{ \alpha } \sum_{ ( k_1, \ldots, k_{ \alpha } ) \in [ \tau ]_d^{ \alpha } } \prod_{ m = 1 }^{ \alpha } c( k_m ) \notag\\
&\geq ( - 1 )^{ \alpha }\sum_{(k_1, \ldots, k_{ \alpha }) \in[\tau]^\alpha}  \prod_{ m = 1 }^{ \alpha } c( k_m ) - \binom{ \alpha }{ 2 } \sum_{ k = 1 }^{ \tau } c( k )^2 \sum_{(k_1, \ldots, k_{ \alpha - 2 }) \in [\tau]^{\alpha-2} } \prod_{ m = 1 }^{ \alpha - 2 } c( k_m ). \label{eq:repeated_terms_lower_bound}
\end{align}

Expanding \eqref{survivor_function_ub} and using \eqref{eq:repeated_terms_lower_bound}, we obtain
\begin{align*}
&\bbP^{ \bfX }(  \bar{ G }_{ \tau_N( \xi, \ell, j; t ) }^{ N, n } = \xi | G_j^{ N, n } = ( \xi, \ell ) ) \\
&= 1 + \sum_{ \alpha = 1 }^{ \tau_N( \xi, \ell, j; t ) - j } \sum_{ j + 1 \leq k_1 < \ldots < k_{ \alpha } \leq \tau_N( \xi, \ell, j; t ) }  ( -1 )^{ \alpha } \\
&\phantom{= 1 + \sum_{ \alpha = 1 }^{ \tau_N( \xi, \ell, j; t ) - j } \sum_{ j + 1 \leq k_1 < \ldots < k_{ \alpha }}} \times \prod_{ m = 1 }^{ \alpha } \bbP^{ \bfX }( | \bar{ G }_{ k_m }^{ N, n } | < | \xi | | \bar{ G }_{ k_m - 1 }^{ N, n } = \xi, G_j^{ N, n } = ( \xi, \ell ) ) \\
&\geq 1 + \sum_{ \alpha = 1 }^{ \tau_N( \xi, \ell, j; t ) - j } \frac{ ( -1 )^{ \alpha } }{ \alpha ! } \binom{ | \xi | }{ 2 }^{ \alpha } \sum_{ k_1, \ldots, k_{ \alpha } = j + 1 }^{ \tau_N( \xi, \ell, j; t ) } \prod_{ m = 1 }^{ \alpha } c_N( \xi, \ell, j; k_m ) \\
&\phantom{\geq 1}\,\, - \sum_{ \alpha = 2 }^{ \tau_N( \xi, \ell, j; t ) - j } \frac{ 1 }{ \alpha ! } \binom{ | \xi | }{ 2 }^{ \alpha } \binom{ \alpha }{ 2 } \sum_{ k = j + 1 }^{ \tau_N( \xi, \ell, j; t ) } c_N( \xi, \ell, j; k )^2  \sum_{ k_1, \ldots, k_{ \alpha - 2 } = j + 1 }^{ \tau_N( \xi, \ell, j; t ) } \prod_{ m = 1 }^{ \alpha - 2 } c_N( \xi, \ell, j; k_m ) \\
&= \sum_{ \alpha = 0 }^{ \tau_N( \xi, \ell, j; t ) - j } \frac{ ( -1 )^{ \alpha } }{ \alpha ! } \binom{ | \xi | }{ 2 }^{ \alpha } \Bigg( \sum_{ k = j + 1 }^{ \tau_N( \xi, \ell, j; t ) }  c_N( \xi, \ell, j; k ) \Bigg)^{ \alpha } \\
&\phantom{\geq 1} - \sum_{ \alpha = 2 }^{ \tau_N( \xi, \ell, j; t ) - j } \frac{ 1 }{ \alpha ! } \binom{ | \xi | }{ 2 }^{ \alpha } \binom{ \alpha }{ 2 } \Bigg( \sum_{ k = j + 1 }^{ \tau_N( \xi, \ell, j; t ) } c_N( \xi, \ell, j; k ) \Bigg)^{ \alpha - 2 } \sum_{ k = j + 1 }^{ \tau_N( \xi, \ell, j; t ) } c_N( \xi, \ell, j; k )^2.
\end{align*}
Using the bounds in \eqref{generalised_inverse_sandwich} and \eqref{elementary_binomial_bound}, 
\begin{align*}
\bbP^{ \bfX }&(  \bar{ G }_{ \tau_N( \xi, \ell, j; t ) }^{ N, n } = \xi | G_j^{ N, n } = ( \xi, \ell ) ) \\
&\geq \sum_{ \alpha = 0, \text{ even} }^{ \tau_N( \xi, \ell, j; t ) - j } \frac{ ( -1 )^{ \alpha } }{ \alpha ! } \binom{ | \xi | }{ 2 }^{ \alpha } t^{ \alpha } \\
&\phantom{\geq} + \sum_{ \alpha = 1, \text{ odd} }^{ \tau_N( \xi, \ell, j; t ) - j } \frac{ ( -1 )^{ \alpha } }{ \alpha ! } \binom{ | \xi | }{ 2 }^{ \alpha } [ t^{ \alpha } + \alpha ( t + 1 )^{ \alpha - 1 } c_N( \xi, \ell, j; \tau_N( \xi, \ell, j; t ) ) ] \\
&\phantom{\geq} - \sum_{ \alpha = 2 }^{ \tau_N( \xi, \ell, j; t ) - j } \frac{ 1 }{ \alpha ! } \binom{ | \xi | }{ 2 }^{ \alpha } \binom{ \alpha }{ 2 } ( t + 1 )^{ \alpha - 2 } \sum_{ k = j + 1 }^{ \tau_N( \xi, \ell, j; t ) } c_N( \xi, \ell, j; k )^2 \\
&\geq \sum_{ \alpha = 0 }^{ \infty } \frac{ ( -t )^{ \alpha } }{ \alpha ! } \binom{ | \xi | }{ 2 }^{ \alpha } \mathds{ 1 }\{ \tau_N( \xi, \ell, j; t ) \geq \alpha + j \} \\
&\phantom{\geq} - c_N( \xi, \ell, j; \tau_N( \xi, \ell, j; t ) ) \binom{ | \xi | }{ 2 } \sum_{ \alpha = 1 }^{ \infty } \frac{ ( t + 1 )^{ \alpha - 1 } }{ ( \alpha - 1 ) ! } \binom{ | \xi | }{ 2 }^{ \alpha - 1 } \\
&\phantom{\geq} - \binom{ | \xi | }{ 2 }^2 \sum_{ \alpha = 2 }^{ \infty } \frac{ 1 }{ \alpha ! } \binom{ | \xi | }{ 2 }^{ \alpha - 2 } \binom{ \alpha }{ 2 } ( t + 1 )^{ \alpha - 2 } \sum_{ k = j + 1 }^{ \tau_N( \xi, \ell, j; t ) } c_N( \xi, \ell, j; k )^2.
\end{align*}
As for the upper bound, by \eqref{kingman_assumption_2}, \eqref{kingman_assumption_3}, and \eqref{indicator_limit_2}, we have that
\begin{equation} \label{exponential_lower}
\lim_{ N \to \infty } \bbE[ \bbP^{ \bfX }(  \bar{ G }_{ \tau_N( \xi, \ell, j; t ) }^{ N, n } = \xi | G_j^{ N, n } = ( \xi, \ell ) ) ] \geq \exp\Bigg( - \binom{ | \xi | }{ 2 } t \Bigg).
\end{equation}

\textbf{Part 2.} Next we show that the merger event is between exactly two uniformly chosen blocks.

Since \eqref{kingman_assumption_4} holds for any $\eta$ such that $\xi \prec \eta$, of which there are $\binom{ | \xi | }{ 2 }$, mergers involving more than two lineages occur at a rate which vanishes as $N \to \infty$.
Also by \eqref{kingman_assumption_4}, the lineages involved in a binary merger are sampled uniformly in the $N \to \infty$ limit:
\begin{align*}
&\lim_{ N \to \infty } \bbE[ \bbP^{ \bfX }( \bar{ G }_{ \tau_N( \xi, \ell, j; t ) }^{ N, n } = \eta, \bar{ G }_{ \tau_N( \xi, \ell, j; t ) - 1 }^{ N, n } = \xi | G_j^{ N, n } = ( \xi, \ell ) ) ]  \\
&= \lim_{ N \to \infty } \bbE[ \bbP^{ \bfX }( \bar{ G }_{ \tau_N( \xi, \ell, j; t ) }^{ N, n } = \eta | \bar{ G }_{ \tau_N( \xi, \ell, j; t ) - 1 }^{ N, n } = \xi, G_j^{ N, n } = ( \xi, \ell ) )  \\
&\phantom{= \lim_{ N \to \infty } \bbE[} \times \bbP^{ \bfX }( \bar{ G }_{ \tau_N( \xi, \ell, j; t ) - 1 }^{ N, n } = \xi | G_j^{ N, n } = ( \xi, \ell ) ) ] \\
&= \lim_{ N \to \infty } \bbE[ c_N( \xi, \ell, j; \tau_N( \xi, \ell, j; t ) ) \bbP^{ \bfX }( \bar{ G }_{ \tau_N( \xi, \ell, j; t ) - 1 }^{ N, n } = \xi | G_j^{ N, n } = ( \xi, \ell ) ) ],
\end{align*}
where the last line is justified by the Dominated Convergence Theorem since the integrand is trivially bounded by one.
The right-hand side does not depend on $\eta$, as required.

\textbf{Part 3.} We now verify that we can concatenate waiting times and mergers to construct the whole process from initial condition $\{ \{ 1 \}, \ldots, \{ n \} \}$ to the most recent common ancestor $\{ \{ 1, \ldots, n \} \}$.

To see that waiting times and mergers can be concatenated to construct a process, let $0 < t_1 < t_2 < \infty$ and $j \in \bbN_0$, and fix $\xi \prec \eta$ with $| \xi | \geq 3$.
Then, by definition of conditional probability,
\begin{align*}
&\sum_{ \ell' \in [ N ]_d^{ | \eta | } } \bbE[ \bbP^{ \bfX }( \bar{ G }_{ \tau_N( \eta, \ell', \tau_N( \xi, \ell, j; t_1 ); t_2 - t_1 ) }^{ N, n } = \eta, G_{ \tau_N( \xi, \ell, j; t_1 ) }^{ N, n } = ( \eta, \ell' ), \\
&\phantom{\bbE[ \bbP^{ \bfX }( \bar{ G }_{ \tau_N( \eta, \ell', \tau_N( \xi, \ell, j; t_1 ); t_2 - t_1 ) }^{ N, n } = \eta, } \mspace{67mu} \bar{ G }_{ \tau_N( \xi, \ell, j; t_1 ) - 1 }^{ N, n } = \xi | G_j^{ N, n } = ( \xi, \ell ) ) ]  \\
&= \sum_{ \ell' \in [ N ]_d^{ | \eta | } } \bbE[ \bbP^{ \bfX }( \bar{ G }_{ \tau_N( \xi, \ell, j; t_1 ) - 1 }^{ N, n } = \xi | G_j^{ N, n } = ( \xi, \ell ) ) \\
&\phantom{=} \times \bbP^{ \bfX }( \bar{ G }_{ \tau_N( \xi, \ell, j; t_1 ) }^{ N, n } = \eta | \bar{ G }_{ \tau_N( \xi, \ell, j; t_1 ) - 1 }^{ N, n } = \xi, G_j^{ N, n } = ( \xi, \ell ) ) \\
&\phantom{=} \times \bbP^{ \bfX }( G_{ \tau_N( \xi, \ell, j; t_1 ) }^{ N, n } = ( \eta, \ell' ) | \bar{ G }_{ \tau_N( \xi, \ell, j; t_1 ) }^{ N, n } = \eta, \bar{ G }_{ \tau_N( \xi, \ell, j; t_1 ) - 1 }^{ N, n } = \xi, G_j^{ N, n } = ( \xi, \ell ) ) \\
&\phantom{=} \times \bbP^{ \bfX }( \bar{ G }_{ \tau_N( \eta, \ell', \tau_N( \xi, \ell, j; t_1 ); t_2 - t_1 ) }^{ N, n } = \eta | G_{ \tau_N( \xi, \ell, j; t_1 ) }^{ N, n } = ( \eta, \ell' ), \bar{ G }_{ \tau_N( \xi, \ell, j; t_1 ) - 1 }^{ N, n } = \xi,  G_j^{ N, n } = ( \xi, \ell ) ) ].
\end{align*}
The sum on the right-hand side is non-negative and bounded by 1 uniformly in $N$.
Hence, we can use \eqref{kingman_assumption_4} to replace its second factor with $c_N(\xi, \ell, j; \tau_N( \xi, \ell, j; t_1 ) )$ without changing the $N \to \infty$ limit.
Moreover, given $\bfX$ and $G_{ \tau_N( \xi, \ell, j; t_1 ) }^{ N, n } = ( \eta, \ell' )$, the past and future of the genealogical process  from time $\tau_N( \xi, \ell, j; t_1 )$ are conditionally independent as can be seen, for example, by a simple D-separation argument \cite[Chapter 3]{pearl:1988}.
Hence, noting that dominated convergence allows us to use \eqref{kingman_assumption_4} inside an expected value,
\begin{align*}
&\sum_{ \ell' \in [ N ]_d^{ | \eta | } } \bbE[ \bbP^{ \bfX }( \bar{ G }_{ \tau_N( \eta, \ell', \tau_N( \xi, \ell, j; t_1 ); t_2 - t_1 ) }^{ N, n } = \eta, G_{ \tau_N( \xi, \ell, j; t_1 ) }^{ N, n } = ( \eta, \ell' ),  \bar{ G }_{ \tau_N( \xi, \ell, j; t_1 ) - 1 }^{ N, n } = \xi | G_j^{ N, n } = ( \xi, \ell ) ) ]  \\
&\sim \bbE\Bigg[ \sum_{ \ell' \in [ N ]_d^{ | \eta | } } \bbP^{ \bfX }( \bar{ G }_{ \tau_N( \xi, \ell, j; t_1 ) - 1 }^{ N, n } = \xi | G_j^{ N, n } = ( \xi, \ell ) ) c_N( \xi, \ell, j; \tau_N( \xi, \ell, j; t_1 ) ) \\
&\phantom{\sim \bbE[}\; \times \bbP^{ \bfX }( G_{ \tau_N( \eta, \ell, j; t_1 ) }^{ N, n } = ( \eta, \ell' ) | \bar{ G }_{ \tau_N( \eta, \ell, j; t_1 ) }^{ N, n } = \eta, \bar{ G }_{ \tau_N( \xi, \ell, j; t_1 ) - 1 }^{ N, n } = \xi, G_j^{ N, n } = ( \xi, \ell ) ) \\
&\phantom{= \bbE\Bigg[} \times \bbP^{ \bfX }( \bar{ G }_{ \tau_N( \eta, \ell', \tau_N( \xi, \ell, j; t_1 ); t_2 - t_1 ) }^{ N, n } = \eta | G_{ \tau_N( \xi, \ell, j; t_1 ) }^{ N, n } = ( \eta, \ell' ) ) \Bigg],
\end{align*}
where $f_N \sim g_N$ means $\lim_{ N \to \infty } f_N / g_N = 1$.
We will now use the Tower Law to condition on $X_{j : \tau_N( \xi, \ell, j; t_1 )}$, along with the fact that all terms on the right-hand side above are $\sigma( X_{j : \tau_N( \xi, \ell, j; t_1 )} )$-measurable except the final one:
\begin{align*}
&\sum_{ \ell' \in [ N ]_d^{ | \eta | } } \bbE[ \bbP^{ \bfX }( \bar{ G }_{ \tau_N( \eta, \ell', \tau_N( \xi, \ell, j; t_1 ); t_2 - t_1 ) }^{ N, n } = \eta, G_{ \tau_N( \xi, \ell, j; t_1 ) }^{ N, n } = ( \eta, \ell' ),  \bar{ G }_{ \tau_N( \xi, \ell, j; t_1 ) - 1 }^{ N, n } = \xi | G_j^{ N, n } = ( \xi, \ell ) ) ]  \\
&\sim \bbE\Bigg[ \sum_{ \ell' \in [ N ]_d^{ | \eta | } } \bbP^{ \bfX }( \bar{ G }_{ \tau_N( \xi, \ell, j; t_1 ) - 1 }^{ N, n } = \xi | G_j^{ N, n } = ( \xi, \ell ) ) c_N( \xi, \ell, j; \tau_N( \xi, \ell, j; t_1 ) ) \\
&\mspace{45mu} \times \bbP^{ \bfX }( G_{ \tau_N( \eta, \ell, j; t_1 ) }^{ N, n } = ( \eta, \ell' ) | \bar{ G }_{ \tau_N( \eta, \ell, j; t_1 ) }^{ N, n } = \eta, \bar{ G }_{ \tau_N( \xi, \ell, j; t_1 ) - 1 }^{ N, n } = \xi, G_j^{ N, n } = ( \xi, \ell ) ) \\
&\mspace{45mu} \times \bbE[ \bbP^{ \bfX }( \bar{ G }_{ \tau_N( \eta, \ell', \tau_N( \xi, \ell, j; t_1 ); t_2 - t_1 ) }^{ N, n } = \eta | G_{ \tau_N( \xi, \ell, j; t_1 ) }^{ N, n } = ( \eta, \ell' ) ) | X_{ j : \tau_N( \xi, \ell, j; t_1 ) } ] \Bigg] \\
&\sim \bbE\Bigg[ \sum_{ \ell' \in [ N ]_d^{ | \eta | } } \bbP^{ \bfX }( \bar{ G }_{ \tau_N( \xi, \ell, j; t_1 ) - 1 }^{ N, n } = \xi | G_j^{ N, n } = ( \xi, \ell ) ) c_N( \xi, \ell, j; \tau_N( \xi, \ell, j; t_1 ) ) \\
&\mspace{45mu} \times \bbP^{ \bfX }( G_{ \tau_N( \eta, \ell, j; t_1 ) }^{ N, n } = ( \eta, \ell' ) | \bar{ G }_{ \tau_N( \eta, \ell, j; t_1 ) }^{ N, n } = \eta, \bar{ G }_{ \tau_N( \xi, \ell, j; t_1 ) - 1 }^{ N, n } = \xi, G_j^{ N, n } = ( \xi, \ell ) ) \\
&\mspace{45mu} \times \bbE[ \bbP^{ \bfX }( \bar{ G }_{ \tau_N( \eta, \ell', \tau_N( \xi, \ell, j; t_1 ); t_2 - t_1 ) }^{ N, n } = \eta | G_{ \tau_N( \xi, \ell, j; t_1 ) }^{ N, n } = ( \eta, \ell' ) ) | X_{ \tau_N( \xi, \ell, j; t_1 ) } ] \Bigg],
\end{align*}
where the last equality again follows by conditional independence.
In Part 1, we have already shown that 
\begin{equation*}
\bbE[ \bbP^{ \bfX }( \bar{ G }_{ \tau_N( \eta, \ell', \tau_N( \xi, \ell, j; t_1 ); t_2 - t_1 ) }^{ N, n } = \eta | G_{ \tau_N( \xi, \ell, j; t_1 ) }^{ N, n } = ( \eta, \ell' ) ) | X_{ \tau_N( \xi, \ell, j; t_1 ) } ] \to \exp\Bigg( - \binom{ | \eta | }{ 2 } ( t_2 - t_1 ) \Bigg)
\end{equation*}
as $N \to \infty$.
Hence, using the Dominated Convergence Theorem to interchange expectations and limits as needed,
\begin{align*}
&\sum_{ \ell' \in [ N ]_d^{ | \eta | } } \bbE[ \bbP^{ \bfX }( \bar{ G }_{ \tau_N( \eta, \ell', \tau_N( \xi, \ell, j; t_1 ); t_2 - t_1 ) }^{ N, n } = \eta, G_{ \tau_N( \xi, \ell, j; t_1 ) }^{ N, n } = ( \eta, \ell' ),  \bar{ G }_{ \tau_N( \xi, \ell, j; t_1 ) - 1 }^{ N, n } = \xi | G_j^{ N, n } = ( \xi, \ell ) ) ]  \\
&\sim \bbE\Bigg[ \bbP^{ \bfX }( \bar{ G }_{ \tau_N( \xi, \ell, j; t_1 ) - 1 }^{ N, n } = \xi | G_j^{ N, n } = ( \xi, \ell ) ) c_N( \xi, \ell, j; \tau_N( \xi, \ell, j; t_1 ) ) \exp\Bigg( - \binom{ | \eta | }{ 2 } ( t_2 - t_1 ) \Bigg) \\
&\mspace{45mu} \times \sum_{ \ell' \in [ N ]_d^{ | \eta | } } \bbP^{ \bfX }( G_{ \tau_N( \eta, \ell, j; t_1 ) }^{ N, n } = ( \eta, \ell' ) | \bar{ G }_{ \tau_N( \eta, \ell, j; t_1 ) }^{ N, n } = \eta, \bar{ G }_{ \tau_N( \xi, \ell, j; t_1 ) - 1 }^{ N, n } = \xi, G_j^{ N, n } = ( \xi, \ell ) ) \Bigg] \\
&= \exp\Bigg( - \binom{ | \eta | }{ 2 } ( t_2 - t_1 ) \Bigg)  \bbE[ \bbP^{ \bfX }( \bar{ G }_{ \tau_N( \xi, \ell, j; t_1 ) - 1 }^{ N, n } = \xi | G_j^{ N, n } = ( \xi, \ell ) ) c_N( \xi, \ell, j; \tau_N( \xi, \ell, j; t_1 ) ) ].
\end{align*}
The $N \to \infty$ convergence 
\begin{equation*}
\bbE[ \bbP^{ \bfX }( \bar{ G }_{ \tau_N( \xi, \ell, j; t_1 ) - 1 }^{ N, n } = \xi | G_j^{ N, n } = ( \xi, \ell ) ) c_N( \xi, \ell, j; \tau_N( \xi, \ell, j; t_1 ) ) ] \to \exp\Bigg( - \binom{ |\xi|}{2} t_1 \Bigg)
\end{equation*}
follows via the argument in Part 1 of the proof.
Iterating this argument yields convergence of finite-dimensional distributions for any fixed number of initial lineages.

Since the marginal ancestral process $( \bar{ G }_{ \tau_N( \xi, \ell, j; t ) - 1 }^{ N, n } )_{t \geq 0}$ undergoes at most $n - 1$ jumps, its finite-dimensional distributions are determined by the holding times between those jumps and the law of which lineages merge in each jump.
Those have been fully characterised in Parts 1--3, showing convergence to the Kingman coalescent in the sense of finite-dimensional distributions.
Part 4 below will complete the proof of weak convergence in the Skorokhod $J_1$ topology.

\textbf{Part 4.} Finally, we control the modulus of continuity to prove weak convergence, using an argument which follows very similarly to \cite[Proof of Theorem 3.1]{mohle1999}.

By \cite[Corollary 7.8, Chapter 3]{ethier/kurtz:1986}, we require relative compactness of the family of genealogical processes indexed by $N$.
To that end, we define the modulus of continuity,
\begin{equation*}
w( \bar{ G }_{ \tau_N( \cdot ) }^{ N, n }, \delta, t ) := \inf_{ | z_i - z_{ i - 1 } | > \delta } \max_i \sup_{ u, v \in [ z_{ i - 1 }, z_i ) } \mathds{ 1 }_{ \{ \bar{ G }_{ \tau_N( u ) }^{ N, n } \neq \bar{ G }_{ \tau_N( v ) }^{ N, n } \} },
\end{equation*}
for $\delta > 0$, $t > 0$, and $0 = z_0 < z_1 < \ldots < z_{ k - 1 } < t < z_j$ for some finite $j$, which exists because of the minimal separation $\delta$.
Since the state space of $\bar{ G }_{ \tau_N( t ) }^{ N, n }$ is finite, by \cite[Corollary 7.4, Chapter 3]{ethier/kurtz:1986} it suffices to show that for every $\eta \in ( 0, 1 )$ and $t > 0$, there exists $\delta > 0$ such that
\begin{equation}\label{eq:modulus_of_continuity}
\liminf_{ N \to \infty } \bbP( w( \bar{ G }_{ \tau_N( \cdot ) }^{ N, n }, \delta, t ) < \eta ) > 1 - \eta.
\end{equation}
This can be done by noticing that $\bar{ G }_{ \tau_N( \cdot ) }^{ N, n }$ jumps at most $n - 1$ times, and we have already shown that its holding times between jumps converge weakly to independent, exponentially distributed random variables (see \eqref{exponential_upper}, \eqref{exponential_lower} in Part 1).

Let $( Z_1, \ldots, Z_{ n - 1 } )$ denote the limiting holding times between jumps, i.e.\ the entries are independent and $Z_j \sim \text{Exp}\big( \binom{ n - j + 1 }{ 2 } \big)$.
They can be seen as the $N \to \infty$ limits of $( T_n, \ldots T_2 )$ introduced above Theorem \ref{thm:main_result}.
All these jumps of the genealogical process are separated by time windows of width $\delta > 0$ with probability 
\begin{equation*}
\prod_{ m = 1 }^{ n - 1 } \bbP( Z_m > \delta ) = \exp\Bigg( - \sum_{ m = 1 }^{ n - 1 } \binom{ n - m + 1 }{ 2 } \delta \Bigg),
\end{equation*}
and this event implies that the modulus of continuity vanishes.
Thus
\begin{equation*}
\liminf_{ N \to \infty } \bbP( w( \bar{ G }_{ \tau_N( \cdot ) }^{ N, n }, \delta, t ) = 0 ) = \exp\Bigg( - \sum_{ m = 1 }^{ n - 1 } \binom{ n - m + 1 }{ 2 } \delta \Bigg),
\end{equation*}
and \eqref{eq:modulus_of_continuity} holds by choosing $\delta \leq - \log( 1 - \eta ) / \sum_{ m = 1 }^{ n - 1 } \binom{ n - m + 1 }{ 2 }$.
\end{proof}

\section{Results for particular resampling schemes}\label{section:resampling_schemes}

In this section we demonstrate that the assumptions of Theorem \ref{thm:main_result} hold for practical resampling schemes under verifiable conditions.
The assumptions we require amount to the so-called \emph{strong mixing condition}, which is standard in the analysis of SMC methods.
Numerical evidence suggests strong mixing is often an unnecessary assumption, but relaxing it presents considerable technical difficulties.

The specific schemes we cover are multinomial and stratified resampling.
The former is arguably the simplest and most analytically tractable scheme, but suboptimal in practice, while the latter is a prototypical example of a superior low-variance scheme.
We expect that convergence could be proven using similar arguments for many standard schemes, such as residual or systematic resampling.

Proofs in this section are largely technical calculations, often to check the assumptions of Theorem \ref{thm:main_result}.
To aid readability, they have been postponed to Appendix~\ref{app:particular}.

\subsection{Multinomial resampling}

Under multinomial resampling, the ancestor indices $( a_k( 1 ), \ldots, a_k( N ) ) | X_k$ in the forward-in-time particle system are conditionally independent given $X_k$.
Each index $a_k(i)$ is sampled independently from the categorical distribution on $[N]$ with probabilities proportional to potentials,
\begin{equation}\label{eq:forward_weights}
a_k( i ) | X_k \sim \text{Categorical}( g_k( X_k( 1 ) ), \ldots, g_k( X_k( N ) ) ),
\end{equation}
with the probabilities parametrising the categorical distribution suitably normalised.
In reverse time, ancestor indices given particle locations $( X_k, X_{ k - 1 } )$ are also conditionally independent, and
\begin{align}
&a_k( i ) | X_k, X_{ k - 1 } \notag \\
&\sim \text{Categorical}( g_k( X_k( 1 ) ) M_k( X_k( 1 ), X_{ k - 1 }( i ) ), \ldots, g_k( X_k( N ) ) M_k( X_k( N ), X_{ k - 1 }( i ) ) ). \label{eq:reverse_weights}
\end{align}
Conditioning on these two generations of locations renders $a_k$ independent from all other generations in either direction of time.
Hence, the conditional distribution of a single reverse-time ancestral lineage given $\bfX$ coincides with that in the \emph{backward simulation} algorithm \cite{godsilletal:2004}.
The joint distribution of several lineages only differs from backward simulation in that our lineages merge together into a common ancestor when they sample the same ancestor index, while those in backward simulation remain distinct lineages which happen to overlap for one generation.

\begin{prop}\label{prop:multinomial}
Suppose there exists $\gamma \geq 1$ and a function $f : \calX \to ( 0, \infty )$ such that
\begin{equation}\label{eq:strong_mixing}
\frac{ f( y ) }{ \gamma } \leq g_k( x ) M_k( x, y ) \leq \gamma f( y )
\end{equation}
for each $x \in \calX$ and $k \geq 1$.
Then multinomial resampling satisfies the conditions of Theorem \ref{thm:main_result}.
Moreover, the timescale $\tau_N(t)$ satisfies the almost sure bounds
\begin{equation}\label{multinomial_timescale}
\Big\lfloor \frac{N t}{\gamma^4} \Big\rfloor \leq \tau_N(t) \leq \frac{(n + 1) n (n - 1) }{6} \lceil \gamma^4 N t \rceil,
\end{equation}
for any initial sample size $n \leq N$.
\end{prop}

Remark 3 of \cite{koskela2020annals} drew a connection between the coalescence probability $c_N( \xi, \ell, k )$ under multinomial resampling and the effective sample size (ESS) of \cite{kong94}:
\begin{equation}\label{eq:ess}
\text{ESS}( k ) := \frac{\big(\sum_{i = 1}^N g_k( X_k( i ) ) \big)^2}{\sum_{ i = 1 }^N g_k( X_k( i ) )^2}.
\end{equation}
The connection of \cite[Remark 3]{koskela2020annals} is only relevant for coalescence probabilities arising from transition probabilities of the form \eqref{eq:neutral_transition_probability}.
However, Proposition \ref{prop:ess} below shows that a similar relationship is valid for our coalescence probabilities and an effective sample size formula in which the backward simulation weights in \eqref{eq:reverse_weights} play the role of the forward weights \eqref{eq:forward_weights} of the standard SMC algorithm.
In order to state the result, let $L$ have conditional law
\begin{equation}\label{labelling_law}
\bbP^{\mathbf{X}}(L = \cdot | \bar{ G }_{ k - 1 }^{ N, n } = \xi, G_j = ( \xi, \ell ) ) = \bbP^{ \bfX }( G_{ k - 1 }^{ N, n } = ( \xi, \cdot ) | \bar{ G }_{ k - 1 }^{ N, n } = \xi, G_j = ( \xi, \ell ) ),
\end{equation}
and let $\bbE^{ \bfX, \xi, \ell }_{ j, k }$ denote a conditional expectation with respect to it.
\begin{prop} \label{prop:ess}
Under multinomial resampling and \eqref{eq:strong_mixing}, as $N \to \infty$ we have
\begin{equation*}
c_N( \xi, \ell, j; k ) = \frac{ 1 + o(1 ) }{ \binom{ | \xi | }{ 2 } } \sum_{ ( v_1, v_2 ) \in [ | \xi | ]_d^2 } \sum_{ m = 1 }^N \bbE_{ j, k }^{ \bfX, \xi, \ell } \Bigg[ \prod_{ i \in \{ 1, 2 \} } \frac{ g_k( X_k( m ) ) M_k( X_k( m ), X_{ k - 1 }( L^{ v_i } ) ) }{ \sum_{ h = 1 }^N g_k( X_k( h ) ) M_k( X_k( h ), X_{ k - 1 }( L^{ v_i } ) ) } \Bigg].
\end{equation*}
\end{prop}

\begin{rem}
Section 4.2 of \cite{brownetal:2021} showed that, under multinomial resampling and the transition probability \eqref{eq:neutral_transition_probability}, genealogical trees of conditional and standard SMC algorithms converge to the same scaling limits under very similar assumptions.
Conditional SMC differs from standard SMC in that one particle has a pre-determined trajectory which is guaranteed to survive all resampling steps with at least one offspring, forming the so-called immortal line.
This modification is essential for correctness of the particle Gibbs algorithm \cite{andrieu2010}, and also finds applications in other areas of Monte Carlo simulation \cite{jacob19, shestopaloff19} and optimisation \cite[Chapter 6]{finke15}.
The key to transferring the convergence proof from standard SMC to the conditional setting was that, under \eqref{eq:strong_mixing}, the probability that any fixed particle is chosen as an ancestor is negligible in the $N \to \infty$ limit.
The same is true in our corrected proof, and hence the presence of the immortal line will have an effect on the genealogy of a finite number of particles with probability tending to zero as $N \to \infty$.
\end{rem}

\subsection{Stratified resampling}

Under stratified resampling, particles are assigned ancestors by decomposing the unit interval into weight-based segments.
Specifically, let $( U_1, U_2, \ldots, U_N )$ be independent, with $U_i \sim U( ( i - 1 ) / N, i / N )$, and let
\begin{equation*}
\bar{ g }_k^{ ( i ) } := \sum_{ m = 1 }^i \frac{ g_k( X_k( m ) ) }{ \sum_{ v = 1 }^N g_k( X_k( v ) ) },
\end{equation*}
with $\bar{ g }_k^{ ( 0 ) } \equiv 0$.
Then $a_k( i ) = m$ if $U_i \in [ \bar{ g }_k^{ ( m - 1 ) }, \bar{ g }_k^{ ( m ) } )$.
The resulting distribution of ancestral indices depends on the ordering of particles, which makes analysis considerably more involved.
Hence, we consider a slight modification in which particle indices are reshuffled uniformly at each iteration, prior to the choice of ancestors.
The same reshuffling was suggested in \cite[page 290]{andrieu2010} to facilitate analysis of particle MCMC methods. 

\begin{prop}\label{prop:stratified}
Suppose there exists $\gamma \geq 1$, a function $f : \calX \to ( 0, \infty )$, and a sequence $\varepsilon_N > 0$ such that
\begin{align}
\frac{ f( y ) }{ \sqrt{ \gamma } } &\leq M_k( x, y ) \leq \sqrt{ \gamma } f( y ), \label{eq:m_mixing} \\
\frac{ 1 }{ \sqrt{ \gamma } } &\leq g_k( x ) \leq \sqrt{ \gamma }, \label{eq:uniform_weight_bound} \\
\lim_{ N \to \infty } &\bbP( \bfX : \max_{ i \in [ N ] }\{ \bar{ g }_k^{ ( i ) } - \bar{g}_k^{ ( i - 1 ) } \} > 1 / N + \varepsilon_N ) = 1, \label{eq:environment}
\end{align}
for each $x \in \calX$ and $k \geq 1$.
Then stratified resampling with uniform shuffling satisfies the conditions of Theorem \ref{thm:main_result}.
Moreover, the timescale $\tau_N(t)$ satisfies
\begin{align}
\Big\lfloor (1 + o(1)) \frac{N t}{\gamma^4} \Big\rfloor &\leq \tau_N(t) \text{ a.s.}, \label{stratified_timescale_lb}\\
\lim_{N \to \infty} \bbP\Bigg( \tau_N(t) \leq \frac{(n + 1) n (n - 1)}{6} \Big\lceil \frac{\gamma^2 N t}{\varepsilon_N} \Big\rceil \Bigg) &\to 1, \label{stratified_timescale_ub}
\end{align}
for any fixed initial sample size $n \leq N$.
\end{prop}

\begin{rem}
The square roots in \eqref{eq:m_mixing} and \eqref{eq:uniform_weight_bound} are superficial.
Noting that \eqref{eq:m_mixing} and \eqref{eq:uniform_weight_bound} imply \eqref{eq:strong_mixing}, the inclusion of the square roots yields bounds that are consistent with \eqref{eq:strong_mixing} using the same constants as in the case of multinomial resampling.
\end{rem}

Condition \eqref{eq:environment} is sufficient to rule out the pathological case where particle weights are too close to equal, resulting in a genealogical process in which mergers never take place.
It could be weakened to requiring $\bar{ g }_k^{ ( i ) } - \bar{g}_k^{ ( i - 1 ) } > 1 / N + \varepsilon_N$ for infinitely many generations, rather than all of them, with no changes to the proof.
The cost of such a relaxation is a more complicated expression in \eqref{stratified_timescale_ub}, which would also need to account for those pre-limiting generations where $\bar{ g }_k^{ ( i ) } - \bar{g}_k^{ ( i - 1 ) } > 1 / N + \varepsilon_N$ is violated and a merger is very unlikely under stratified resampling.

Since \eqref{eq:environment} is a statement about the distribution of particle weights, it is typically difficult to verify.
However, in any given generation it is implied by bounding the effective sample size \eqref{eq:ess} away from $N$.

\begin{prop}\label{prop:environment}
For each $k \geq 1$ and $N \in \mathbb{N}$, if ESS$(k) \leq N (1 - \eta_N)$ for some $\eta_N \in (0, 1)$, then there is an $i \in [N]$ such that $\bar{ g }_k^{ ( i ) } - \bar{g}_k^{ ( i - 1 ) } > 1 / N + \varepsilon_N$ holds for some $\varepsilon_N$ satisfying
\begin{equation*}
\varepsilon_N \geq \frac{1}{N(N - 1)} \sqrt{\frac{\eta_N}{1 - \eta_N}}.
\end{equation*}
\end{prop}

\subsection{Stochastic rounding}

Stochastic roundings are a class of resampling mechanisms for which
\begin{equation*}
\bbP( \nu_k( i ) = m | g_k ) =
\begin{cases}
1 - \frac{ N g_k( X_k( i ) ) }{ \sum_{ v = 1 }^N g_k( X_k( v ) ) } + \Big\lfloor \frac{ N g_k( X_k( i ) ) }{ \sum_{ v = 1 }^N g_k( X_k( v ) ) } \Big\rfloor &\text{if } m = \Big\lfloor \frac{ N g_k( X_k( i ) ) }{ \sum_{ v = 1 }^N g_k( X_k( v ) ) } \Big \rfloor, \\
\frac{ N g_k( X_k( i ) ) }{ \sum_{ v = 1 }^N g_k( X_k( v ) ) } - \Big\lfloor \frac{ N g_k( X_k( i ) ) }{ \sum_{ v = 1 }^N g_k( X_k( v ) ) } \Big\rfloor &\text{if } m = \Big\lfloor \frac{ N g_k( X_k( i ) ) }{ \sum_{ v = 1 }^N g_k( X_k( v ) ) } \Big\rfloor + 1.
\end{cases}
\end{equation*}
There are a number of example resampling mechanisms with these marginals, including systematic resampling \cite{carpenter1999, whitley1994}, the branching system of \cite{crisan1997}, and Srinivasan sampling \cite{gerber2017}.
We conjecture that a proof akin to that of Proposition \ref{prop:stratified} could be produced for at least some of these schemes, particularly systematic resampling which closely resembles its stratified counterpart.
However, each such proof is likely to be a similarly lengthy and technical calculation, and hence we do not pursue them here.

Remark 4.4 of \cite{brownetal:2021} argued that the expected merger rates of all stochastic rounding schemes coincide.
The basis for the argument was an explicit conditional merger probability obtained from \eqref{eq:neutral_transition_probability} which only depended on marginal moments of family sizes.
Since \eqref{eq:neutral_transition_probability} is only valid for systems whose ancestral index vectors are independent in different generations, the scope of \cite[Remark 4.4]{brownetal:2021} must be subject to the same constraint.

\subsection{Ordering of merger probabilities}\label{merger_ordering}

Appendix A of \cite{brownetal:2021} showed that when conditional merger probabilities can be written in the form in \eqref{eq:neutral_transition_probability}, resampling schemes based on stochastic rounding dominate multinomial resampling in that the quantity
\begin{equation}\label{neutral_expected_coalescence_probability}
\frac{ 1 }{ ( N )_2 } \sum_{ i = 1 }^N \bbE^{ \bfX }[ ( \nu_k( i ) )_2 ]
\end{equation}
is always at least as large under the latter as under the former.
The relevance of the result is that timescales for convergence to the Kingman coalescent are typically obtained as generalised inverses of \eqref{neutral_expected_coalescence_probability}, so that a resampling scheme for which \eqref{neutral_expected_coalescence_probability} is higher will exhibit faster coalescence.
The relevance to SMC algorithms is that the speed of coalescence can be thought of as the rate of build-up of path degeneracy arising from distinct lineages merging into common ancestors.

Our timescale is slightly different, obtained as the generalised inverse of coalescence probabilities in \eqref{kingman_timescale}.
Here we present a counterexample which shows that ordering of second factorial moments in \eqref{neutral_expected_coalescence_probability} does not imply ordering of coalescence probabilities in \eqref{kingman_timescale}.
Hence, there does not appear to be a simple way to compare the speed of coalescence between resampling schemes using our timescale.
Our counterexample is between stratified and multinomial resampling, but systematic resampling can be substituted in place of stratified resampling in all of the computations below as well.

Consider $N = 4$ and suppose we are tracking $n = 3$ lineages.
Suppose that the ordered, normalised weights of each of the four available parents in the next generation are $(1 - 3z, z, z, z)$, for $z \in [0, 1/3]$, and that $M(x, y) \propto 1$ so that forward and backward-in-time weights coincide.
Figure \ref{fig:stratified} illustrates these ancestral weights, and how they line up with the four intervals used for allocating children in stratified resampling.
\begin{figure}[!ht]
\centering
	\begin{tikzpicture}
		\draw (0, 2.5) -- (0, 2) -- (1.25, 2) node[above]{1/4} -- (2.5, 2) -- (3.75, 2) node[above]{1/4} -- (5, 2) -- (6.25, 2) node[above]{1/4} -- (7.5, 2) -- (8.75, 2) node[above]{1/4} -- (10, 2) -- (10, 2.5) node[below right]{Stratification};
		\draw(2.5, 2) -- (2.5, 2.5);
		\draw(5, 2) -- (5, 2.5);
		\draw(7.5, 2) -- (7.5, 2.5);
		\draw (0, 1.5) -- (0, 1) -- (4.4, 1) node[above]{$1 - 3z$} -- (8.8, 1) -- (9, 1) node[above]{$z$} -- (9.2, 1) -- (9.4, 1) node[above]{$z$} -- (9.6, 1) -- (9.8, 1) node[above]{$z$} -- (10, 1) -- (10, 1.5) node[below right]{Parent weights};
		\draw(8.8, 1) -- (8.8, 1.5);
		\draw(9.2, 1) -- (9.2, 1.5);
		\draw(9.6, 1) -- (9.6, 1.5);
	\end{tikzpicture}
\caption{The ordered weights of four parents lined up against the four stratification intervals used to implement stratified resampling.}
\label{fig:stratified}
\end{figure}

 Let $\bbP_M, \bbP_S, \bbE_M$, and $\bbE_S$ denote probabilities and expectations under multinomial and stratified resampling, respectively.
Then, given the weights in Figure \ref{fig:stratified} with $z \leq 1/12$ and with $j < k \in \bbN$,
\begin{equation*}
\bbP_S^{ \bfX }( | \bar{ G }_k^{ N, n } | < 3 | | \bar{ G }_{ k - 1 }^{ N, n } | = 3, G_j^{ N, n } ) = 1,
\end{equation*}
because there is no way to allocate the three lineages to the four stratification intervals without at least three merging in the first parent.
This is the case irrespective of whether particles are shuffled prior to resampling.
For multinomial resampling,
\begin{align*}
\bbP_M^{ \bfX }( | \bar{ G }_k^{ N, n } | < 3 | | \bar{ G }_{ k - 1 }^{ N, n } | = 3, G_j^{ N, n } ) &=  ( 1 - 3 z )^3 + 3 z^3 + 3 \binom{ 3 }{ 2 } ( 1 - 3 z )^2 z + 3 \binom{ 3 }{ 2 } z^2 ( 1 - z )\\
&= 1 - 18z^2 + 48z^3 \\
&\leq \bbP_S^{ \bfX }( | \bar{ G }_k^{ N, n } | < 3 | | \bar{ G }_{ k - 1 }^{ N, n } | = 3, G_j^{ N, n } ),
\end{align*}
with equality only if $z = 0$, so that the probability of at least one coalescence is higher under stratified resampling.

The factorial moments in \eqref{neutral_expected_coalescence_probability} satisfy the opposite inequality, in that the ones under multinomial resampling dominate those under stratified resampling.
Using the fact that, for $X \sim \text{Bin}(n, p)$,
\begin{equation*}
\bbE[ ( X )_2 ] = \text{Var}( X ) + \bbE[ X ]^2 - \bbE[ X ] = n p ( 1 - p ) + n^2 p^2 - n p = ( n )_2 p^2,
\end{equation*}
we have
\begin{align}
\bbE_S^{ \bfX }[ ( \nu_k( 1 ) )_2 ] &= ( 4 )_2 \frac{ 1 / 4 - 3 z }{ 1 / 4 } + ( 3 )_2 \frac{ 3 z }{ 1 / 4 } = 12 - 72 z, \label{stratified_1}\\
\bbE_M^{ \bfX }[ ( \nu_k( 1 ) )_2 ] &= ( 4 )_2 ( 1 - 3z )^2 = 12 - 72z + 108z^2, \notag \\
\bbE_S^{ \bfX }[ ( \nu_k( i ) )_2 ] &= 0 \text{ for }  i \in \{ 2, 3, 4 \}, \label{stratified_2}\\
\bbE_M^{ \bfX }[ ( \nu_k( i ) )_2 ] &= (4)_2 z^2 = 12 z^2 \text{ for }  i \in \{ 2, 3, 4 \}. \notag
\end{align}
Hence, factorial moments of family sizes are always larger under multinomial resampling, in line with the result of \cite{brownetal:2021}.

The factorial moment \eqref{stratified_2} remains correct even if particles are shuffled prior to resampling.
In contrast, under shuffled stratified resampling \eqref{stratified_1} is accurate if particle 1 is randomised to position 1 or 4.
If it ends up in either middle position then
\begin{align*}
&\bbE_S^{ \bfX }[ ( \nu_k( 1 ) )_2 | \text{Particle } 1 \text{ shuffled to position } 2 \text{ or } 3 ] \\
&= ( 4 )_2 \frac{ 1 / 4 - z }{ 1 / 4 } \frac{ 1 / 4 - 2 z }{ 1 / 4 } + ( 3 )_2 \Big(1 - \frac{ 1 / 4 - z }{ 1 / 4 } \frac{ 1 / 4 - 2 z }{ 1 / 4 }  - \frac{z}{1/4} \frac{2z}{1/4} \Big)+ ( 2 )_2 \frac{z}{1/4} \frac{2z}{1/4} \\
&= 12 - 72 z + 64 z^2.
\end{align*}
Because all four shuffling positions are equally likely, the shuffled factorial moment is
\begin{equation*}
\bbE_S^{\bfX}[( \nu_k( 1 ) )_2] = 12 - 72 z + 32 z^2,
\end{equation*}
which is intermediate between unshuffled stratified resampling and multinomial resampling.

\appendix
\section{Results for Particular Resampling Schemes} \label{app:particular}

\subsection{Multinomial resampling} \label{app:multinomial}

\begin{proof}[Proof of Proposition \ref{prop:multinomial}]

The conditional merger probability under multinomial resampling is
\begin{align*}
&\bbP^{ \bfX }( | \bar{ G }_k^{ N, n } | < | \xi | | \bar{ G }_{ k - 1 }^{ N, n } = \xi, G_j^{ N, n } = ( \xi, \ell_0 ) ) \\
&= \sum_{ \ell_1 \in [ N ]_d^{ | \xi | } } \bbP^{ \bfX }( | \bar{ G }_k^{ N, n } | < | \xi | | G_{ k - 1 }^{ N, n } = ( \xi, \ell_1 ), G_j^{ N, n } = ( \xi, \ell_0 ) ) \\
&\phantom{= \sum_{ \ell_1 \in [ N ]_d^{ | \xi | } } } \times \bbP^{ \bfX }( G_{ k - 1 }^{ N, n } = ( \xi, \ell_1 ) | \bar{ G }_{ k - 1 }^{ N, n } = \xi, G_j^{ N, n } = ( \xi, \ell_0 ) ) \\
&= \sum_{ \ell_1 \in [ N ]_d^{ | \xi | } } \bbP^{ \bfX }( | \bar{ G }_k^{ N, n } | < | \xi | | G_{ k - 1 }^{ N, n } = ( \xi, \ell_1 ) ) \bbP^{ \bfX }( G_{ k - 1 }^{ N, n } = ( \xi, \ell_1 ) | \bar{ G }_{ k - 1 }^{ N, n } = \xi,  G_j^{ N, n } = ( \xi, \ell_0 ) ),
\end{align*}
using conditional independence of generation $k$ from generations $\{ j, \ldots, k - 1 \}$ given $( \bfX, \ell_1 )$.
Substituting in the categorical probabilities that lineages $\ell_1^{ j_1 }$ and $\ell_1^{ j_2 }$ merge into a common ancestor, we obtain
\begin{align}
&\bbP^{ \bfX }( | \bar{ G }_k^{ N, n } | < | \xi | | \bar{ G }_{ k - 1 }^{ N, n } = \xi, G_j^{ N, n } = ( \xi, \ell_0 ) ) \notag\\
&\leq \sum_{ \ell_1 \in [ N ]_d^{ | \xi | } } \bbP^{ \bfX }( G_{ k - 1 }^{ N, n } = ( \xi, \ell_1 ) | \bar{ G }_{ k - 1 }^{ N, n } = \xi, G_j^{ N, n } = ( \xi, \ell_0 ) ) \notag \\
&\phantom{ = \sum_{ \ell_1 \in [ N ]_d^{ | \xi | } } } \times \sum_{ ( v_1, v_2 ) \in [ | \xi | ]_{ d, u }^2 } \sum_{ m = 1 }^N \prod_{ i \in \{ 1, 2 \} } \frac{ g_k( X_k( m ) ) M_k( X_k( m ), X_{ k - 1 }( \ell_1^{ v_i } ) ) }{ \sum_{ h = 1 }^N g_k( X_k( h ) ) M_k( X_k( h ), X_{ k - 1 }( \ell_1^{ v_i } ) ) }, \label{multinomial_ub}
\end{align}
which is an inequality because the right-hand side counts some mergers many times.
Using \eqref{eq:strong_mixing}, we have
\begin{align}
&\bbP^{ \bfX }( | \bar{ G }_k^{ N, n } | < | \xi | | \bar{ G }_{ k - 1 }^{ N, n } = \xi, G_j^{ N, n } = ( \xi, \ell_0 ) ) \notag \\
&\leq \sum_{ \ell_1 \in [ N ]_d^{ | \xi | } } \bbP^{ \bfX }( G_{ k - 1 }^{ N, n } = ( \xi, \ell_1 ) | \bar{ G }_{ k - 1 }^{ N, n } = \xi, G_j^{ N, n } = ( \xi, \ell_0 ) ) \sum_{ ( v_1, v_2 ) \in [ | \xi | ]_{ d, u }^2 } \sum_{ m = 1 }^N\prod_{ i \in \{ 1, 2 \} } \frac{ \gamma^2 }{ N } \notag \\
&= \binom{ | \xi | }{ 2 } \frac{ \gamma^4 }{ N }, \label{eq:binary_ub}
\end{align}
because the middle step amounts to a sum of the probability mass function $\bbP^{ \bfX }( G_{ k - 1 }^{ N, n } = ( \xi, \cdot ) | \bar{ G }_{ k - 1 }^{ N, n } = \xi, G_j^{ N, n } = ( \xi, \ell_0 ) )$ over its support $\ell_1 \in [ N ]_d^{ | \xi | }$.
Hence, $c_N( \xi, \ell, j; k ) \to 0$ as $N \to \infty$ uniformly in $k \in \bbN$, so that \eqref{kingman_assumption_1} and \eqref{kingman_assumption_2} hold.
Furthermore, \eqref{kingman_assumption_3} holds because \eqref{generalised_inverse_sandwich} and \eqref{eq:binary_ub} imply
\begin{equation*}
\sum_{ k = j + 1 }^{ \tau_N( \xi, \ell, j; t ) } c_N( \xi, \ell, j; k )^2 \leq \frac{ \gamma^4 }{ N } \sum_{ k = j + 1 }^{ \tau_N( \xi, \ell, j; t ) } c_N( \xi, \ell, j; k ) \leq \frac{ \gamma^4 ( t + 1 ) }{ N },
\end{equation*}
noting that the factor of $\binom{|\xi|}{2}$ in \eqref{eq:binary_ub} cancels with that in the definition of $c_N$ in \eqref{kingman_timescale}.

We will obtain \eqref{kingman_assumption_4} by checking a two-part condition which implies it.
The first is that there exists a sequence $C_N^n \to 0$ as $N \to \infty$ such that
\begin{align} 
&\bbP^{ \bfX }( | \bar{ G }_k^{ N, n } | < | \xi | - 1 | \bar{ G }_{ k - 1 }^{ N, n } = \xi, G_j^{ N, n } = ( \xi, \ell ) ) \notag \\
&\leq C_N^n \bbP^{ \bfX }( | \bar{ G }_k^{ N, n } | = | \xi | - 1 | \bar{ G }_{ k - 1 }^{ N, n } = \xi, G_j^{ N, n } = ( \xi, \ell ) ) ), \label{eq:large_mergers_vanish} 
\end{align}
for any sufficiently large $N$ and almost every $\bfX$, where $C_N^n$ does not depend on $k \geq 0$, $\ell \in [ N ]_d^{ | \xi | }$, or $\bfX$.
The second is that, for every $\xi \prec \eta$,
\begin{align}
&\bbP^{ \bfX }( \bar{ G }_{ \tau_N( \xi, \ell, j; t ) }^{ N, n } = \eta | \bar{ G }_{ \tau_N( \xi, \ell, j; t ) - 1 }^{ N, n } = \xi, G_j^{ N, n } = ( \xi, \ell ) ) \notag \\
&\sim \binom{ | \xi | }{ 2 }^{ -1 } \bbP^{ \bfX }( | \bar{ G }_{ \tau_N( \xi, \ell, j; t ) }^{ N, n } | = | \xi | - 1  | \bar{ G }_{ \tau_N( \xi, \ell, j; t ) - 1 }^{ N, n } = \xi, G_j^{ N, n } = ( \xi, \ell ) ), \label{eq:binary_uniformity}
\end{align}
almost surely as $N \to \infty$, uniformly in $t > 0$, $\ell \in [ N ]_d^{ | \xi | }$, $| \xi | \leq n$, and $\bfX$.
When \eqref{eq:large_mergers_vanish} and \eqref{eq:binary_uniformity} hold,
\begin{align*}
&\bbP^{ \bfX }( \bar{ G }_{ \tau_N( \xi, \ell, j; t ) }^{ N, n } = \eta | \bar{ G }_{ \tau_N( \xi, \ell, j; t ) - 1 }^{ N, n } = \xi, G_j^{ N, n } = ( \xi, \ell ) ) \\
&\sim \binom{ | \xi | }{ 2 }^{ -1 } \bbP^{ \bfX }( | \bar{ G }_{ \tau_N( \xi, \ell, j; t ) }^{ N, n } | = | \xi | - 1  | \bar{ G }_{ \tau_N( \xi, \ell, j; t ) - 1 }^{ N, n } = \xi, G_j^{ N, n } = ( \xi, \ell ) ) \\
&\sim \binom{ | \xi | }{ 2 }^{ -1 } [ \bbP^{ \bfX }( | \bar{ G }_{ \tau_N( \xi, \ell, j; t ) }^{ N, n } | = | \xi | - 1 | \bar{ G }_{ \tau_N( \xi, \ell, j; t ) - 1 }^{ N, n } = \xi, G_j^{ N, n } = ( \xi, \ell ) ) \\
&\phantom{= \binom{ | \xi | }{ 2 }^{ -1 } [}+ \bbP^{ \bfX }( | \bar{ G }_{ \tau_N( \xi, \ell, j; t ) }^{ N, n } | < | \xi | - 1 | \bar{ G }_{ \tau_N( \xi, \ell, j; t ) - 1 }^{ N, n } = \xi, G_j^{ N, n } = ( \xi, \ell ) ) ] \\
&= \binom{ | \xi | }{ 2 }^{ -1 } \bbP^{ \bfX }( | \bar{ G }_{ \tau_N( \xi, \ell, j; t ) }^{ N, n } | < | \xi | | \bar{ G }_{ \tau_N( \xi, \ell, j; t ) - 1 }^{ N, n } = \xi, G_j^{ N, n } = ( \xi, \ell ) ) \\
&= c_N( \xi, \ell, j; \tau_N( \xi, \ell, j; t ) ),
\end{align*}
so that \eqref{kingman_assumption_4} holds as well.

We begin by establishing \eqref{eq:large_mergers_vanish}. 
In order for the number of lineages to decrease by more than one in a single time-step, there must be at least one merger with three (or more) lineages, or at least two mergers involving pairs of lineages.
The corresponding conditional probability given $\bfX$ can be bounded above:
\begin{align*}
&\bbP^{ \bfX }( | \bar{ G }_k^{ N, n } | < | \xi | - 1 | \bar{ G }_{ k - 1 }^{ N, n } = \xi, G_j^{ N, n } = ( \xi, \ell ) ) \\
&= \sum_{ \ell_1 \in [ N ]_d^{ | \xi | } } \bbP^{ \bfX }( | \bar{ G }_k^{ N, n } | < | \xi | - 1 | G_{ k - 1 }^{ N, n } = ( \xi, \ell_1 ) ) \bbP^{ \bfX }( G_{ k - 1 }^{ N, n } = ( \xi, \ell_1 ) | \bar{ G }_{ k - 1 }^{ N, n } = \xi,  G_j^{ N, n } = ( \xi, \ell_0 ) ) \\
&\leq \sum_{ \ell_1 \in [ N ]_d^{ | \xi | } } \bbP^{ \bfX }( G_{ k - 1 }^{ N, n } = ( \xi, \ell_1 ) | \bar{ G }_{ k - 1 }^{ N, n } = \xi, G_j^{ N, n } = ( \xi, \ell_0 ) ) \\
&\phantom{ \leq } \times \Bigg[ \sum_{ ( v_1, v_2, v_3 ) \in [ | \xi | ]_{ d, u }^3 } \sum_{ m = 1 }^N \prod_{ i = 1 }^3 \frac{ g_k( X_k( m ) ) M_k( X_k( m ), X_{ k - 1 }( \ell_1^{ v_i } ) ) }{ \sum_{ h = 1 }^N g_k( X_k( h ) ) M_k( X_k( h ), X_{ k - 1 }( \ell_1^{ v_i } ) ) } \\
&\phantom{ \leq \times \Bigg[ } + \frac{ 1 }{ 8 } \sum_{ ( v_1, \ldots, v_4 ) \in [ | \xi | ]_d^4 } \prod_{ u \in \{ 0, 2 \} } \Bigg( \sum_{ m = 1 }^N \prod_{ i = 1 }^2 \frac{ g_k( X_k( m ) ) M_k( X_k( m ), X_{ k - 1 }( \ell_1^{ v_{ i + u } } ) ) }{ \sum_{ h = 1 }^N g_k( X_k( h ) ) M_k( X_k( h ), X_{ k - 1 }( \ell_1^{ v_{ i + u } } ) ) } \Bigg) \Bigg],
\end{align*}
where the $1/8$ on the last line compensates for the four orderings of $( v_1, v_2, v_3, v_4)$ which leave the $\{ v_1, v_2 \}$ and $\{ v_3, v_4 \}$ mergers unchanged, as well as the two further orderings of the two parent indices.
Using \eqref{eq:strong_mixing}, we have
\begin{align}
&\bbP^{ \bfX }( | \bar{ G }_k^{ N, n } | < | \xi | - 1 | \bar{ G }_{ k - 1 }^{ N, n } = \xi, G_j^{ N, n } = ( \xi, \ell ) ) \notag \\
&\leq \sum_{ \ell_1 \in [ N ]_d^{ | \xi | } } \bbP^{ \bfX }( G_{ k - 1 }^{ N, n } = ( \xi, \ell_1 ) | \bar{ G }_{ k - 1 }^{ N, n } = \xi, G_j^{ N, n } = ( \xi, \ell_0 ) ) \notag \\
&\phantom{ \leq } \times \Bigg[ \frac{ ( | \xi | - 2 ) \gamma^2 }{ 3 N } \sum_{ ( v_1, v_2 ) \in [ | \xi | ]_{ d, u }^2 } \sum_{ m = 1 }^N \prod_{ i = 1 }^2 \frac{ g_k( X_k( m ) ) M_k( X_k( m ), X_{ k - 1 }( \ell_1^{ v_i } ) ) }{ \sum_{ h = 1 }^N g_k( X_k( h ) ) M_k( X_k( h ), X_{ k - 1 }( \ell_1^{ v_i } ) ) } \notag \\
&\phantom{ \leq \times } + \frac{ ( | \xi | - 2 ) ( | \xi | - 3 ) \gamma^4 }{ 4 N }  \sum_{ ( v_1, v_2 ) \in [ | \xi | ]_{ d, u }^2 } \sum_{ m = 1 }^N \prod_{ i = 1 }^2 \frac{ g_k( X_k( m ) ) M_k( X_k( m ), X_{ k - 1 }( \ell_1^{ v_i } ) ) }{ \sum_{ h = 1 }^N g_k( X_k( h ) ) M_k( X_k( h ), X_{ k - 1 }( \ell_1^{ v_i } ) ) } \Bigg] \notag \\
&\leq \frac{ ( | \xi | - 2 ) \gamma^2 }{ N } \Bigg[ \frac{ 1 }{ 3 } + \frac{ ( | \xi | - 3 ) \gamma^2 }{ 4 } \Bigg]  \sum_{ \ell_1 \in [ N ]_d^{ | \xi | } } \bbP^{ \bfX }( G_{ k - 1 }^{ N, n } = ( \xi, \ell_1 ) | \bar{ G }_{ k - 1 }^{ N, n } = \xi, G_j^{ N, n } = ( \xi, \ell_0 ) ) \notag \\
&\phantom{ \leq \frac{ ( | \xi | - 2 ) \gamma^2 }{ N } } \times \sum_{ ( v_1, v_2 ) \in [ | \xi | ]_{ d, u }^2 } \sum_{ m = 1 }^N \prod_{ i = 1 }^2 \frac{ g_k( X_k( m ) ) M_k( X_k( m ), X_{ k - 1 }( \ell_1^{ v_i } ) ) }{ \sum_{ h = 1 }^N g_k( X_k( h ) ) M_k( X_k( h ), X_{ k - 1 }( \ell_1^{ v_i } ) ) }, \label{eq:multinomial_big_mergers}
\end{align}
where the $1/3$ compensates for the three ways to choose an index among $(v_1, v_2, v_3) \in [|\xi|]_{d, u}^2$, and the $1/4 = 2/8$ for the two ways to choose a family among $\{v_1, v_2\}$ and $\{v_3, v_4\}$.
Still by \eqref{eq:strong_mixing},
\begin{align*}
\frac{ g_k( X_k( m ) ) M_k( X_k( m ), X_{ k - 1 }( \ell_1^{ v_i } ) ) }{ \sum_{ h = 1 }^N g_k( X_k( h ) ) M_k( X_k( h ), X_{ k - 1 }( \ell_1^{ v_i } ) ) } \leq \gamma^4 \frac{ g_k( X_k( m ) ) M_k( X_k( m ), X_{ k - 1 }( \ell_1^i ) ) }{ \sum_{ h = 1 }^N g_k( X_k( h ) ) M_k( X_k( h ), X_{ k - 1 }( \ell_1^i ) ) }.
\end{align*}
for any $v_i, i \in [ | \xi | ]$.
Applying this bound to both factors in \eqref{eq:multinomial_big_mergers} yields
\begin{align}
&\bbP^{ \bfX }( | \bar{ G }_k^{ N, n } | < | \xi | - 1 | \bar{ G }_{ k - 1 }^{ N, n } = \xi, G_j^{ N, n } = ( \xi, \ell ) ) \notag \\
&\leq \frac{ \gamma^{ 10 } }{ N } \Bigg[ \binom{ | \xi | }{ 3 }+ \binom{ | \xi | }{ 2 } \binom{  | \xi | - 2 }{ 2 } \frac{ \gamma^2 }{ 2 } \Bigg]  \sum_{ \ell_1 \in [ N ]_d^{ | \xi | } } \bbP^{ \bfX }( G_{ k - 1 }^{ N, n } = ( \xi, \ell_1 ) | \bar{ G }_{ k - 1 }^{ N, n } = \xi, G_j^{ N, n } = ( \xi, \ell_0 ) ) \notag \\
&\phantom{ \leq \sum_{ \ell_1 \in [ N ]_d^{ | \xi | } } } \times \sum_{ m = 1 }^N \prod_{ i = 1 }^2 \frac{ g_k( X_k( m ) ) M_k( X_k( m ), X_{ k - 1 }( \ell_1^i ) ) }{ \sum_{ h = 1 }^N g_k( X_k( h ) ) M_k( X_k( h ), X_{ k - 1 }( \ell_1^i ) ) } \notag \\
&\leq \frac{ \gamma^{ 10 } }{ N } \Bigg[ \binom{ | \xi | }{ 3 }+ \binom{ | \xi | }{ 2 } \binom{  | \xi | - 2 }{ 2 } \frac{ \gamma^2 }{ 2 } \Bigg] \bbP^{ \bfX }( | \bar{ G }_k^{ N, n } | < | \xi | | \bar{ G }_{ k - 1 }^{ N, n } = \xi, G_j^{ N, n } = ( \xi, \ell ) ), \label{eq:lower_order}
\end{align}
where the final inequality holds because the factor on the second line of the middle step is the probability that lineages $\ell_1^1$ and $\ell_1^2$ merge.
Hence, \eqref{eq:large_mergers_vanish} holds.

To see that \eqref{eq:binary_uniformity} holds, note that \eqref{eq:strong_mixing} implies that the conditional law of the ancestor index of a lineage can be written as
\begin{equation*}
\bbP^{ \bfX }( a_k( i ) = m ) = \frac{ g_k( X_k( m ) ) M_k( X_k( m ), X_{ k - 1 }( i ) ) }{ \sum_{ h = 1 }^N g_k( X_k( h ) ) M_k( X_k( h ), X_{ k - 1 }( i ) ) } = \frac{ 1 }{ \gamma^2 N } + \frac{ \gamma^2 - 1 }{ \gamma^2 } r( i, m ),
\end{equation*}
where $r( i, \cdot )$ is a probability mass function on $[ N ]$.
Hence, a lineage can sample its ancestor by first flipping a Bernoulli($1 / \gamma^2$)-distributed coin.
If the flip succeeds, the ancestor is sampled uniformly, while if the flip fails, the ancestor is sampled from the remainder mass function $r( i, \cdot )$.
The number of generations until $| \xi |$ simultaneous successes is geometrically distributed with parameter $1 / \gamma^{ 2 | \xi | }$.
Let $\Xi$ denote the event that there is at least one such generation strictly between the initial generation $j$ and generation $\tau_N( \xi, \ell, j; t )$.
By \eqref{eq:binary_ub}, we have that
\begin{equation*}
\tau_N( \xi, \ell, j; t ) \geq \lfloor  N t / \gamma^4 \rfloor,
\end{equation*}
almost surely.
Then, as $N \to \infty$,
\begin{equation*}
\bbP^{ \bfX }( \Xi ) \geq 1 - \Bigg(1 - \frac{ 1 }{ \gamma^{ 2 | \xi | } } \Bigg)^{ \lfloor N t / \gamma^4 \rfloor - j - 1 } \to 1.
\end{equation*}
For $k \in \{j + 1, \ldots, \tau_N(\xi, \ell, j; t) - 1\}$, let $\Xi_k$ denote the event that generation $k$ is the first one in which $|\xi|$ simultaneous successes happen.
On the event $\Xi$, the left-hand side of \eqref{eq:binary_uniformity} can be decomposed into contributions from each $\Xi_k$ as
\begin{align}
&\bbP^{ \bfX }( \bar{ G }_{ \tau_N( \xi, \ell, j; t ) }^{ N, n } = \eta | \bar{ G }_{ \tau_N( \xi, \ell, j; t ) - 1 }^{ N, n } = \xi, G_j^{ N, n } = ( \xi, \ell ), \Xi ) \notag \\
&= \sum_{k = j + 1}^{\tau_N(\xi, \ell, j; t) - 1} \sum_{\ell' \in [N]_d^{|\xi|}} \bbP^{\bfX}(\bar{ G }_{ \tau_N( \xi, \ell, j; t ) }^{ N, n } = \eta | \bar{ G }_{ \tau_N( \xi, \ell, j; t ) - 1 }^{ N, n } = \xi, G_k^{N, n} = (\xi, \ell'), \Xi_k) \notag \\
&\mspace{166mu} \times \bbP^{\bfX}(\Xi_k, G_k^{ N, n } = ( \xi, \ell' ) | \bar{ G }_{ \tau_N( \xi, \ell, j; t ) - 1 }^{ N, n } = \xi, G_j^{ N, n } = ( \xi, \ell ), \Xi) \notag \\
&= \sum_{k = j + 1}^{\tau_N(\xi, \ell, j; t) - 1} \sum_{\ell' \in [N]_d^{|\xi|}} \bbP^{\bfX}(\bar{ G }_{ \tau_N( \xi, \ell, j; t ) }^{ N, n } = \eta | \bar{ G }_{ \tau_N( \xi, \ell, j; t ) - 1 }^{ N, n } = \xi, G_k^{N, n} = (\xi, \ell'), \Xi_k) \notag \\
&\mspace{166mu} \times \frac{\bbP^{\bfX}(\bar{ G }_{ \tau_N( \xi, \ell, j; t ) - 1 }^{ N, n } = \xi | G_k^{ N, n } = ( \xi, \ell' ), \Xi_k)}{\bbP^{\bfX}(\bar{ G }_{ \tau_N( \xi, \ell, j; t ) - 1 }^{ N, n } = \xi | G_j^{ N, n } = ( \xi, \ell ), \Xi)} \notag \\
&\mspace{166mu} \times \bbP^{\bfX}(\Xi_k, G_k^{ N, n } = ( \xi, \ell' ) | G_j^{ N, n } = ( \xi, \ell ), \Xi) \notag \\
&= \sum_{k = j + 1}^{\tau_N(\xi, \ell, j; t) - 1} \sum_{\ell' \in [N]_d^{|\xi|}} \frac{\bbP^{\bfX}(\bar{ G }_{ \tau_N( \xi, \ell, j; t ) }^{ N, n } = \eta, \bar{ G }_{ \tau_N( \xi, \ell, j; t ) - 1 }^{ N, n } = \xi | G_k^{N, n} = (\xi, \ell'), \Xi_k)}{\bbP^{\bfX}(\bar{ G }_{ \tau_N( \xi, \ell, j; t ) - 1 }^{ N, n } = \xi | G_j^{ N, n } = ( \xi, \ell ), \Xi)}\notag \\
&\mspace{166mu} \times \frac{(1 - \gamma^{-2 |\xi|})^{k - j - 1} \gamma^{-2|\xi|}}{N^{|\xi|} (1 - (1 - \gamma^{-2 |\xi|})^{\tau_N(\xi, \ell, j;t) - j - 1})}, \label{eq:label_mixing}
\end{align}
where the first equality follows from conditioning on $G_k^{N, n}$ and $\Xi_k$, the Markov property, and the fact that $\Xi_k \subset \Xi$.
The second equality follows from Bayes' rule, and the third substitutes in the geometric mass functions for $\Xi_k$ and $\Xi$, as well as $N^{-|\xi|}$ as the probability of choosing the $|\xi|$ parents with labels $\ell'$ in generation $k$ uniformly at random.
To see that the right-hand side is independent of $\eta$, note that given $\bfX$, the probability of any particular merger depends only on the block labels in the generation prior to the merger.
Since \eqref{eq:label_mixing} is averaged over uniformly distributed generation-$k$ particle labels $\ell'$, and hence also all particle labels between generations $k + 1, \ldots, \tau_N(\xi, \ell, j; t) - 1$, the dependence on $\eta$ in the numerator is superficial: the right-hand side is invariant to replacing $\eta$ with any other $\eta' \neq \eta : \xi \prec \eta'$.
Thus the left-hand side of \eqref{eq:binary_uniformity} does not depend on $\eta$ on an event with probability converging to one as $N \to \infty$, and hence \eqref{eq:binary_uniformity} holds.

To check \eqref{timescale_assumption} we will obtain a lower bound on $c_N( \xi, \ell, j; k )$.
By considering the probability that two specific particles merge and using \eqref{eq:strong_mixing},
\begin{align}
&\binom{ | \xi | }{ 2 }^{ -1 } \bbP^{ \bfX }( | \bar{ G }_k^{ N, n } | < | \xi | | \bar{ G }_{ k - 1 }^{ N, n } = \xi, G_j^{ N, n } = ( \xi, \ell_0 ) ) \notag \\
&\geq \binom{ | \xi | }{ 2 }^{ -1 } \sum_{ \ell_1 \in [ N ]_d^{ | \xi | } } \bbP^{ \bfX }( G_{ k - 1 }^{ N, n } = ( \xi, \ell_1 ) | \bar{ G }_{ k - 1 }^{ N, n } = \xi, G_j^{ N, n } = ( \xi, \ell_0 ) ) \notag \\
&\phantom{ = \binom{ | \xi | }{ 2 }^{ -1 } \sum_{ \ell_1 \in [ N ]_d^{ | \xi | } } } \times \sum_{ m = 1 }^N \prod_{ i \in \{ 1, 2 \} } \frac{ g_k( X_k( m ) ) M_k( X_k( m ), X_{ k - 1 }( \ell_1^i ) ) }{ \sum_{ h = 1 }^N g_k( X_k( h ) ) M_k( X_k( h ), X_{ k - 1 }( \ell_1^i ) ) } \label{multinomial_lb} \\
&\geq \binom{ | \xi | }{ 2 }^{ -1 } \sum_{ \ell_1 \in [ N ]_d^{ | \xi | } } \bbP^{ \bfX }( G_{ k - 1 }^{ N, n } = ( \xi, \ell_1 ) | \bar{ G }_{ k - 1 }^{ N, n } = \xi, G_j^{ N, n } = ( \xi, \ell_0 ) ) \sum_{ m = 1 }^N \frac{ 1 }{ N^2 \gamma^4 } \notag \\
&= \binom{ | \xi | }{ 2 }^{ -1 } \frac{ 1 }{ N \gamma^4 }. \notag
\end{align}
Hence, $c_N( \xi, \ell, j; k )$ is bounded away from zero for each $N$ uniformly in $\ell$, $j$, and $k$, which means $\tau_N( \xi, \ell, j; t )$ has an almost surely finite upper bound and $\bbP( \tau_N( \xi, \ell, j; t ) = \infty ) = 0$ for any $t \in ( 0, \infty )$.

Finally, the bounds on the timescale $\tau_N(t)$ in \eqref{multinomial_timescale} follow immediately from \eqref{eq:binary_ub} and \eqref{multinomial_lb}, as well as the fact that
\begin{equation}\label{binomial_sum}
\sum_{j = 2}^n \binom{j}{2} = \frac{(n + 1) n (n - 1)}{6}.
\end{equation}
\end{proof}

\begin{proof}[Proof of Proposition \ref{prop:ess}]

By \eqref{multinomial_ub} and with the random vector $L$ defined as in \eqref{labelling_law}, we have
\begin{align}
&c_N( \xi, \ell, j; k ) \label{prop3.2_ub} \\
&\leq \frac{1}{\binom{ | \xi | }{ 2 }} \sum_{ \ell_1 \in [ N ]_d^{ | \xi | } } \bbP^{ \bfX }( G_{ k - 1 }^{ N, n } = ( \xi, \ell_1 ) | \bar{ G }_{ k - 1 }^{ N, n } = \xi, G_j^{ N, n } = ( \xi, \ell ) ) \notag \\
&\phantom{ \leq \frac{1}{\binom{ | \xi | }{ 2 }} \sum_{ \ell_1 \in [ N ]_d^{ | \xi | } } } \times \sum_{ ( v_1, v_2 ) \in [ | \xi | ]_{ d, u }^2 }\sum_{ m = 1 }^N \prod_{ i \in \{ 1, 2 \} } \frac{ g_k( X_k( m ) ) M_k( X_k( m ), X_{ k - 1 }( \ell_1^{ v_i } ) ) }{ \sum_{ h = 1 }^N g_k( X_k( h ) ) M_k( X_k( h ), X_{ k - 1 }( \ell_1^{ v_i } ) ) } \notag \\
&= \frac{1}{\binom{ | \xi | }{ 2 }} \sum_{ ( v_1, v_2 ) \in [ | \xi | ]_{ d, u }^2 } \sum_{ m = 1 }^N \bbE_{ j, k }^{ \bfX, \xi, \ell } \Bigg[ \prod_{ i \in \{ 1, 2 \} } \frac{ g_k( X_k( m ) ) M_k( X_k( m ), X_{ k - 1 }( L^{ v_i } ) ) }{ \sum_{ h = 1 }^N g_k( X_k( h ) ) M_k( X_k( h ), X_{ k - 1 }( L^{ v_i } ) ) } \Bigg], \notag
\end{align}
so that the right-hand side matches the statement of Proposition \ref{prop:ess}.

In the other direction, \eqref{multinomial_lb} yields the same expression as a lower bound when $| \xi | = 2$, but for general $\xi$ we require a sharper bound.
To that end,
\begin{align*}
& c_N( \xi, \ell, j; k ) \\
&= \frac{ 1 }{ \binom{ | \xi | }{ 2 } } \bbE_{ j, k }^{ \bfX, \xi, \ell }\Bigg[ 1 -  \sum_{ ( m_1, \ldots, m_{ | \xi | } ) \in [ N ]_d^{ | \xi | } } \prod_{ i = 1 }^{ | \xi | } \frac{ g_k( X_k( m_i ) ) M_k( X_k( m_i ), X_{ k - 1 }( L^i ) ) }{ \sum_{ h = 1 }^N g_k( X_k( h ) ) M_k( X_k( h ), X_{ k - 1 }( L^i ) ) } \Bigg] \\
&\geq \frac{ 1 }{ \binom{ | \xi | }{ 2 } } \bbE_{ j, k }^{ \bfX, \xi, \ell }\Bigg[ 1  - \Bigg( 1 - \sum_{ ( v_1, v_2 ) \in [ | \xi | ]_{ d, u }^2 } \sum_{ m = 1 }^N \Bigg[\prod_{ i \in \{ 1, 2 \} } \frac{ g_k( X_k( m ) ) M_k( X_k( m ), X_{ k - 1 }( L^{ v_i } ) ) }{ \sum_{ h = 1 }^N g_k( X_k( h ) ) M_k( X_k( h ), X_{ k - 1 }( L^{ v_i } ) ) } \Bigg] \\
&\phantom{\geq  \frac{ 1 }{ \binom{ | \xi | }{ 2 } } } \times \sum_{ \substack{ ( m_i : i \in [ | \xi | ] \setminus \{ v_1, v_2 \} ) \in [ N ]_d^{ | \xi | - 2 } \\ \text{every } m_i \neq m } } \prod_{ i \in [ | \xi | ] \setminus \{ v_1, v_2 \} } \frac{ g_k( X_k( m_i ) ) M_k( X_k( m_i ), X_{ k - 1 }( L^i ) ) }{ \sum_{ h = 1 }^N g_k( X_k( h ) ) M_k( X_k( h ), X_{ k - 1 }( L^i ) ) } \Bigg) \Bigg].
\end{align*}
The sum on the last line can be bounded below by using Lemma \ref{moehle_lb_lemma} in Appendix~\ref{app:lemmas} with $r = | \xi | - 2$ and 
\begin{equation*}
a( m, \ell ) = \frac{ g_k( X_k( m ) ) M_k( X_k( m ), X_{ k - 1 }( \ell ) ) }{ \sum_{ h = 1 }^N g_k( X_k( h ) ) M_k( X_k( h ), X_{ k - 1 }( \ell ) ) },
\end{equation*}
which yields
\begin{align}
&c_N( \xi, \ell, j; k ) \label{prop3.2_lb} \\
&\geq \frac{ 1 }{ \binom{ | \xi | }{ 2 } } \sum_{ ( v_1, v_2 ) \in [ | \xi | ]_{ d, u }^2 } \sum_{ m = 1 }^N \bbE_{ j, k }^{ \bfX, \xi, \ell }\Bigg[\prod_{ i \in \{ 1, 2 \} } \frac{ g_k( X_k( m ) ) M_k( X_k( m ), X_{ k - 1 }( L^{ v_i } ) ) }{ \sum_{ h = 1 }^N g_k( X_k( h ) ) M_k( X_k( h ), X_{ k - 1 }( L^{ v_i } ) ) } \Bigg] \notag \\
&\phantom{\geq}- \frac{ 1 }{ \binom{ | \xi | }{ 2 } } \sum_{ ( v_1, v_2, v_3 ) \in [ | \xi | ]_{ d, u }^3 } \sum_{ m = 1 }^N \bbE_{ j, k }^{ \bfX, \xi, \ell }\Bigg[\prod_{ i \in \{ 1, 2, 3 \} } \frac{ g_k( X_k( m ) ) M_k( X_k( m ), X_{ k - 1 }( L^{ v_i } ) ) }{ \sum_{ h = 1 }^N g_k( X_k( h ) ) M_k( X_k( h ), X_{ k - 1 }( L^{ v_i } ) ) } \Bigg] \notag \\
&\phantom{\geq}- \frac{ 1 }{ \binom{ | \xi | }{ 2 } } \sum_{ ( v_1, \ldots, v_4 ) \in [ | \xi | ]_d^4 } \frac{1}{8}  \bbE_{ j, k }^{ \bfX, \xi, \ell }\Bigg[ \prod_{ u \in \{ 0, 2 \} } \sum_{ m = 1 }^N \prod_{ i \in \{ 1, 2 \} } \frac{ g_k( X_k( m ) ) M_k( X_k( m ), X_{ k - 1 }( L^{ v_{ i + u } } ) ) }{ \sum_{ h = 1 }^N g_k( X_k( h ) ) M_k( X_k( h ), X_{ k - 1 }( L^{ v_{ i + u } } ) ) } \Bigg], \notag
\end{align}
where the factor of $1/8$ on the last line compensates for the fact that four ordered 4-tuples $( v_1, v_2, v_3, v_4 )$ correspond to the same double-merger between $\{ v_1, v_2\}$ and $\{ v_3, v_4 \}$, and that there are also two corresponding orderings of the parent indices $m$.
By \eqref{eq:strong_mixing},
\begin{align}
&\bbE_{ j, k }^{ \bfX, \xi, \ell }\Bigg[\prod_{ i \in \{ 1, 2, 3 \} } \frac{ g_k( X_k( m ) ) M_k( X_k( m ), X_{ k - 1 }( L^{ v_i } ) ) }{ \sum_{ h = 1 }^N g_k( X_k( h ) ) M_k( X_k( h ), X_{ k - 1 }( L^{ v_i } ) ) } \Bigg] \notag \\
&\leq \frac{\gamma^2}{N} \bbE_{ j, k }^{ \bfX, \xi, \ell }\Bigg[\prod_{ i \in \{ 1, 2 \} } \frac{ g_k( X_k( m ) ) M_k( X_k( m ), X_{ k - 1 }( L^{ v_i } ) ) }{ \sum_{ h = 1 }^N g_k( X_k( h ) ) M_k( X_k( h ), X_{ k - 1 }( L^{ v_i } ) ) } \Bigg], \label{lot_1} \\
&\bbE_{ j, k }^{ \bfX, \xi, \ell }\Bigg[ \prod_{ u \in \{ 0, 2 \} } \sum_{ m = 1 }^N \prod_{ i \in \{ 1, 2 \} } \frac{ g_k( X_k( m ) ) M_k( X_k( m ), X_{ k - 1 }( L^{ v_{ i + u } } ) ) }{ \sum_{ h = 1 }^N g_k( X_k( h ) ) M_k( X_k( h ), X_{ k - 1 }( L^{ v_{ i + u } } ) ) } \Bigg] \notag \\
&\leq \frac{\gamma^2}{N} \sum_{ m = 1 }^N \bbE_{ j, k }^{ \bfX, \xi, \ell }\Bigg[ \prod_{ i \in \{ 1, 2 \} } \frac{ g_k( X_k( m ) ) M_k( X_k( m ), X_{ k - 1 }( L^{ v_i } ) ) }{ \sum_{ h = 1 }^N g_k( X_k( h ) ) M_k( X_k( h ), X_{ k - 1 }( L^{ v_i } ) ) } \Bigg]. \label{lot_2}
\end{align}
Substituting \eqref{lot_1} and \eqref{lot_2} into the lower bound \eqref{prop3.2_lb} shows that it matches the upper bound in \eqref{prop3.2_ub} in the $N \to \infty$ limit, completing the proof.
\end{proof}

\subsection{Stratified resampling}\label{app:stratified}

\begin{proof}[Proof of Proposition \ref{prop:stratified}]
Under stratified resampling, the conditional probability given $\bfX$ that particle $\ell \in [N]$ in generation $k - 1$ has parent $m \in [N]$ in generation $k$ is
\begin{equation*}
\frac{ | [ \frac{ \ell - 1 }{ N }, \frac{ \ell }{ N } ) \cap [ \bar{ g }_k^{ ( m - 1 ) }, \bar{ g }_k^{ ( m ) } ) | M_k( X_k( m ), X_{ k - 1 }( \ell ) ) }{ \sum_{ h = 1 }^N | [ \frac{ \ell - 1 }{ N }, \frac{ \ell }{ N } ) \cap [ \bar{ g }_k^{ ( h - 1 ) }, \bar{ g }_k^{ ( h ) } ) | M_k( X_k( h ), X_{ k - 1 }( \ell ) ) }.
\end{equation*}
Thus, the conditional probability of at least one merger in generation $k$ among $| \xi |$ lineages can be bounded above by conditioning on the uniformly sampled indices $( h_1, \ldots, h_{ | \xi | } )$ to which they are shuffled, and multiplying by the probability that at least two of those indices obtain the same parent in the stratified resampling scheme:
\begin{align*}
&\bbP^{ \bfX }( | \bar{ G }_k^{ N, n } | < | \xi | | \bar{ G }_{ k - 1 }^{ N, n } = \xi, G_j^{ N, n } = ( \xi, \ell_0 ) ) \\
&\leq \sum_{ \ell_1 \in [ N ]_d^{ | \xi | } } \bbP^{ \bfX }( G_{ k - 1 }^{ N, n } = ( \xi, \ell_1 ) | \bar{ G }_{ k - 1 }^{ N, n } = \xi, G_j^{ N, n } = ( \xi, \ell_0 ) ) \sum_{ ( h_1, \ldots, h_{ | \xi | } ) \in [ N ]_d^{ | \xi | } }  \frac{ 1 }{ ( N )_{ | \xi | } } \\
&\phantom{=} \times \sum_{ ( v_1, v_2 ) \in [ | \xi | ]_{ d, u }^2 } \sum_{ m = 1 }^N \prod_{ i \in \{ 1, 2 \} } \frac{ | [ \frac{ h_{ v_i } - 1 }{ N }, \frac{ h_{ v_i } }{ N } ) \cap [ \bar{ g }_k^{ ( m - 1 ) }, \bar{ g }_k^{ ( m ) } ) | M_k( X_k( m ), X_{ k - 1 }( \ell_1^{ v_i } ) ) }{ \sum_{ r = 1 }^N | [ \frac{ h_{ v_i } - 1 }{ N }, \frac{ h_{ v_i } }{ N } ) \cap [ \bar{ g }_k^{ ( r - 1 ) }, \bar{ g }_k^{ ( r ) } ) | M_k( X_k( r ), X_{ k - 1 }( \ell_1^{ v_i } ) ) } \\
&= \sum_{ \ell_1 \in [ N ]_d^{ | \xi | } } \bbP^{ \bfX }( G_{ k - 1 }^{ N, n } = ( \xi, \ell_1 ) | \bar{ G }_{ k - 1 }^{ N, n } = \xi, G_0^{ N, n } = ( \xi, \ell_0 ) ) \sum_{ ( v_1, v_2 ) \in [ | \xi | ]_{ d, u }^2 } \sum_{ ( h_{ v_1 }, h_{ v_2 } ) \in [ N ]_d^2 } \\
&\phantom{=} \frac{ 1 }{ ( N )_2 } \sum_{ m = 1 }^N \prod_{ i \in \{ 1, 2 \} } \frac{ | [ \frac{ h_{ v_i } - 1 }{ N }, \frac{ h_{ v_i } }{ N } ) \cap [ \bar{ g }_k^{ ( m - 1 ) }, \bar{ g }_k^{ ( m ) } ) | M_k( X_k( m ), X_{ k - 1 }( \ell_1^{ v_i } ) ) }{ \sum_{ r = 1 }^N | [ \frac{ h_{ v_i } - 1 }{ N }, \frac{ h_{ v_i } }{ N } ) \cap [ \bar{ g }_k^{ ( r - 1 ) }, \bar{ g }_k^{ ( r ) } ) | M_k( X_k( r ), X_{ k - 1 }( \ell_1^{ v_i } ) ) },
\end{align*}
where the inequality arises because the right-hand side overcounts some mergers.
Applying \eqref{eq:m_mixing} to the last line, we obtain
\begin{align}
&\bbP^{ \bfX }( | \bar{ G }_k^{ N, n } | < | \xi | | \bar{ G }_{ k - 1 }^{ N, n } = \xi, G_j^{ N, n } = ( \xi, \ell_0 ) ) \notag \\
&\leq \frac{ \gamma^2 }{ ( N )_2 } \sum_{ \ell_1 \in [ N ]_d^{ | \xi | } } \bbP^{ \bfX }( G_{ k - 1 }^{ N, n } = ( \xi, \ell_1 ) | \bar{ G }_{ k - 1 }^{ N, n } = \xi, G_j^{ N, n } = ( \xi, \ell_0 ) ) \notag \\
&\phantom{= \sum_{ \ell_1 \in [ N ]_d^{ | \xi | } } } \times \sum_{ ( v_1, v_2 ) \in [ | \xi | ]_{ d, u }^2 } \sum_{ ( h_{ v_1 }, h_{ v_2 } ) \in [ N ]_d^2 } \sum_{ m = 1 }^N \prod_{ i \in \{ 1, 2 \} } \frac{ | [ \frac{ h_{ v_i } - 1 }{ N }, \frac{ h_{ v_i } }{ N } ) \cap [ \bar{ g }_k^{ ( m - 1 ) }, \bar{ g }_k^{ ( m ) } ) | }{ \sum_{ r = 1 }^N | [ \frac{ h_{ v_i } - 1 }{ N }, \frac{ h_{ v_i } }{ N } ) \cap [ \bar{ g }_k^{ ( r - 1 ) }, \bar{ g }_k^{ ( r ) } ) | } \notag \\
&\leq \frac{ \gamma^2 }{ ( N )_2 } \binom{ | \xi | }{ 2 } \sum_{ \ell_1 \in [ N ]_d^{ | \xi | } } \bbP^{ \bfX }( G_{ k - 1 }^{ N, n } = ( \xi, \ell_1 ) | \bar{ G }_{ k - 1 }^{ N, n } = \xi, G_j^{ N, n } = ( \xi, \ell_0 ) ) \notag \\
&\phantom{\leq \frac{ \gamma^4 }{ ( N )_2 } \binom{ | \xi | }{ 2 } \sum_{ \ell_1 \in [ N ]_d^{ | \xi | } } } \times \sum_{ ( h_1, h_2 ) \in [ N ]_d^2 } \sum_{ m = 1 }^N \prod_{ i \in \{ 1, 2 \} } \frac{ | [ \frac{ h_i - 1 }{ N }, \frac{ h_i }{ N } ) \cap [ \bar{ g }_k^{ ( m - 1 ) }, \bar{ g }_k^{ ( m ) } ) | }{ \sum_{ r = 1 }^N | [ \frac{ h_i - 1 }{ N }, \frac{ h_i }{ N } ) \cap [ \bar{ g }_k^{ ( r - 1 ) }, \bar{ g }_k^{ ( r ) } ) | }. \label{sum_rearrangement}
\end{align}
For fixed $h_i$, the sum in the denominator on the last line simplifies to
\begin{equation}\label{eq:stratified_denominator}
\sum_{ r = 1 }^N  \Big| \Big[ \frac{ h_i - 1 }{ N }, \frac{ h_i }{ N } \Big) \cap [ \bar{ g }_k^{ ( r - 1 ) }, \bar{ g }_k^{ ( r ) } ) \Big| = \Big| \Big[ \frac{ h_i - 1 }{ N }, \frac{ h_i }{ N } \Big) \Big| = \frac{ 1 }{ N },
\end{equation}
so that
\begin{align}
&\bbP^{ \bfX }( | \bar{ G }_k^{ N, n } | < | \xi | | \bar{ G }_{ k - 1 }^{ N, n } = \xi, G_j^{ N, n } = ( \xi, \ell_0 ) ) \notag \\
&\leq \sum_{ \ell_1 \in [ N ]_d^{ | \xi | } } \bbP^{ \bfX }( G_{ k - 1 }^{ N, n } = ( \xi, \ell_1 ) | \bar{ G }_{ k - 1 }^{ N, n } = \xi, G_j^{ N, n } = ( \xi, \ell_0 ) ) \notag \\
&\phantom{ \leq \sum_{ \ell_1 \in [ N ]_d^{ | \xi | } } } \times \frac{ N \gamma^2 }{ N - 1 } \binom{ | \xi | }{ 2 } \sum_{ ( h_1, h_2 ) \in [ N ]_d^2 } \sum_{ m = 1 }^N \prod_{ i \in \{ 1, 2 \} } \Big| \Big[ \frac{ h_i - 1 }{ N }, \frac{ h_i }{ N } \Big) \cap [ \bar{ g }_k^{ ( m - 1 ) }, \bar{ g }_k^{ ( m ) } ) \Big|. \label{eq:stratified_ub_1}
\end{align}
For a fixed $m$, we have the bound
\begin{align}
&\sum_{ ( h_1, h_2 ) \in [ N ]_d^2 } \prod_{ i \in \{ 1, 2 \} } \Big| \Big[ \frac{ h_i - 1 }{ N }, \frac{ h_i }{ N } \Big) \cap [ \bar{ g }_k^{ ( m - 1 ) }, \bar{ g }_k^{ ( m ) } ) \Big| \notag \\
&\leq \Bigg( \sum_{ h = 1 }^N \Big| \Big[ \frac{ h - 1 }{ N }, \frac{ h }{ N } \Big) \cap [ \bar{ g }_k^{ ( m - 1 ) }, \bar{ g }_k^{ ( m ) } ) \Big| \Bigg)^2 = | [ \bar{ g }_k^{ ( m - 1 ) }, \bar{ g }_k^{ ( m ) } ) |^2 \leq \frac{ \gamma^2 }{ N^2 }, \label{eq:stratified_numerator}
\end{align}
where the final inequality follows from \eqref{eq:uniform_weight_bound}.
Substituting back into \eqref{eq:stratified_ub_1},
\begin{align}
&\bbP^{ \bfX }( | \bar{ G }_k^{ N, n } | < | \xi | | \bar{ G }_{ k - 1 }^{ N, n } = \xi, G_j^{ N, n } = ( \xi, \ell_0 ) ) \notag \\
&\leq \sum_{ \ell_1 \in [ N ]_d^{ | \xi | } } \bbP^{ \bfX }( G_{ k - 1 }^{ N, n } = ( \xi, \ell_1 ) | \bar{ G }_{ k - 1 }^{ N, n } = \xi, G_j^{ N, n } = ( \xi, \ell_0 ) ) \frac{ N \gamma^2 }{ N - 1 } \frac{ \gamma^2 }{ N } \binom{ | \xi | }{ 2 } \notag \\
&= \frac{ \gamma^4 ( 1 + o( 1 ) ) }{ N } \binom{ | \xi | }{ 2 }, \label{eq:stratified_merger_ub}
\end{align}
where the sum over $[ N ]_d^{ | \xi | }$ evaluates to one because it adds up a conditional probability mass function over its support.
Similarly to the multinomial case, this upper bound on the binary merger probability means that \eqref{kingman_assumption_1}, \eqref{kingman_assumption_2}, and \eqref{kingman_assumption_3} hold.

As with multinomial resampling, \eqref{eq:large_mergers_vanish} and \eqref{eq:binary_uniformity} are a sufficient condition for \eqref{kingman_assumption_4}.
To show that \eqref{eq:large_mergers_vanish} holds, we write the probability of the number of blocks decreasing by more than one as
\begin{align*}
&\bbP^{ \bfX }( | \bar{ G }_k^{ N, n } | < | \xi | - 1 | \bar{ G }_{ k - 1 }^{ N, n } = \xi, G_j^{ N, n } = ( \xi, \ell ) ) \\
&= \sum_{ \ell_1 \in [ N ]_d^{ | \xi | } } \bbP^{ \bfX }( | \bar{ G }_k^{ N, n } | < | \xi | - 1 | G_{ k - 1 }^{ N, n } = ( \xi, \ell_1 ) )  \bbP^{ \bfX }( G_{ k - 1 }^{ N, n } = ( \xi, \ell_1 ) | \bar{ G }_{ k - 1 }^{ N, n } = \xi, G_j^{ N, n } = ( \xi, \ell_0 ) ) \\
&\leq \sum_{ \ell_1 \in [ N ]_d^{ | \xi | } } \bbP^{ \bfX }( G_{ k - 1 }^{ N, n } = ( \xi, \ell_1 ) | \bar{ G }_{ k - 1 }^{ N, n } = \xi, G_j^{ N, n } = ( \xi, \ell_0 ) )  \sum_{ ( h_1, \ldots, h_{ | \xi | } ) \in [ N ]_d^{ | \xi | } }  \frac{ 1 }{ ( N )_{ | \xi | } } \\
&\phantom{\leq} \times \sum_{ ( v_1, v_2 ) \in [ | \xi | ]_{ d, u }^2 } \sum_{ m = 1 }^N \Bigg( \prod_{ i = 1 }^2 \frac{ | [ \frac{ h_{ v_i } - 1 }{ N }, \frac{ h_{ v_i } }{ N } ) \cap [ \bar{ g }_k^{ ( m - 1 ) }, \bar{ g }_k^{ ( m ) } ) | M_k( X_k( m ), X_{ k - 1 }( \ell_1^{ v_i } ) ) }{ \sum_{ r = 1 }^N | [ \frac{ h_{ v_i } - 1 }{ N }, \frac{ h_{ v_i } }{ N } ) \cap [ \bar{ g }_k^{ ( r - 1 ) }, \bar{ g }_k^{ ( r ) } ) | M_k( X_k( r ), X_{ k - 1 }( \ell_1^{ v_i } ) ) } \Bigg) \\
&\phantom{\leq} \times \Bigg[ \frac{ 1 }{ 3 } \sum_{ v_3 \neq v_1, v_2 } \frac{ | [ \frac{ h_{ v_3 } - 1 }{ N }, \frac{ h_{ v_3 } }{ N } ) \cap [ \bar{ g }_k^{ ( m - 1 ) }, \bar{ g }_k^{ ( m ) } ) | M_k( X_k( m ), X_{ k - 1 }( \ell_1^{ v_3 } ) ) }{ \sum_{ r = 1 }^N | [ \frac{ h_{ v_3 } - 1 }{ N }, \frac{ h_{ v_3 } }{ N } ) \cap [ \bar{ g }_k^{ ( r - 1 ) }, \bar{ g }_k^{ ( r ) } ) | M_k( X_k( r ), X_{ k - 1 }( \ell_1^{ v_3 } ) ) } + \frac{ 1 }{ 2 } \\
&\phantom{\times \Bigg[} \times \sum_{ \substack{ ( v_3, v_4 ) \in [ | \xi | ]_{ d, u }^2 \\ \text{both } \neq v_1, v_2 } } \sum_{ m' \neq m } \prod_{ i = 3 }^4 \frac{ | [ \frac{ h_{ v_i } - 1 }{ N }, \frac{ h_{ v_i } }{ N } ) \cap [ \bar{ g }_k^{ ( m' - 1 ) }, \bar{ g }_k^{ ( m' ) } ) | M_k( X_k( m' ), X_{ k - 1 }( \ell_1^{ v_i } ) ) }{ \sum_{ r = 1 }^N | [ \frac{ h_{ v_i } - 1 }{ N }, \frac{ h_{ v_i } }{ N } ) \cap [ \bar{ g }_k^{ ( r - 1 ) }, \bar{ g }_k^{ ( r ) } ) | M_k( X_k( r ), X_{ k - 1 }( \ell_1^{ v_i } ) ) } \Bigg],
\end{align*}
where the right-hand side is an upper bound on the probabilities of a merger of three lineages, or two simultaneous mergers of a pair of lineages each.
All larger mergers must contain at least one of these events.
Applying \eqref{eq:m_mixing} and \eqref{eq:stratified_denominator} to the terms in square brackets and simplifying the sum over $( h_1, \ldots, h_{ | \xi | } )$ similarly to \eqref{sum_rearrangement}, 
\begin{align}
&\bbP^{ \bfX }( | \bar{ G }_k^{ N, n } | < | \xi | - 1 | \bar{ G }_{ k - 1 }^{ N, n } = \xi, G_j^{ N, n } = ( \xi, \ell ) ) \notag \\
&\leq \sum_{ \ell_1 \in [ N ]_d^{ | \xi | } } \bbP^{ \bfX }( G_{ k - 1 }^{ N, n } = ( \xi, \ell_1 ) | \bar{ G }_{ k - 1 }^{ N, n } = \xi, G_j^{ N, n } = ( \xi, \ell_0 ) ) \sum_{ ( v_1, v_2 ) \in [ | \xi | ]_{ d, u }^2 } \frac{ 1 }{ ( N )_2 } \notag \\
&\phantom{\leq} \times \sum_{ ( h_1, h_2 ) \in [ N ]_d^2 } \sum_{ m = 1 }^N \Bigg( \prod_{ i = 1 }^2 \frac{ | [ \frac{ h_i - 1 }{ N }, \frac{ h_i }{ N } ) \cap [ \bar{ g }_k^{ ( m - 1 ) }, \bar{ g }_k^{ ( m ) } ) | M_k( X_k( m ), X_{ k - 1 }( \ell_1^{ v_i } ) ) }{ \sum_{ r = 1 }^N | [ \frac{ h_i - 1 }{ N }, \frac{ h_i }{ N } ) \cap [ \bar{ g }_k^{ ( r - 1 ) }, \bar{ g }_k^{ ( r ) } ) | M_k( X_k( r ), X_{ k - 1 }( \ell_1^{ v_i } ) ) } \Bigg) \notag \\
&\phantom{\leq \times } \times \Bigg[ \frac{ N \gamma  ( | \xi | - 2 ) }{ 3 ( N - 2 ) } \sum_{ h_3 \neq h_1, h_2 } \Big| \Big[ \frac{ h_3 - 1 }{ N }, \frac{ h_3 }{ N } \Big) \cap [ \bar{ g }_k^{ ( m - 1 ) }, \bar{ g }_k^{ ( m ) } ) \Big| \notag \\
&\phantom{\leq \times \times \Bigg[} + \frac{ ( N \gamma )^2 }{ 2 ( N - 2 )_2 } \binom{ | \xi | - 2 }{ 2 }  \sum_{ \substack{ ( h_3, h_4 ) \in [ N ]_d^2 \\ \text{both } \neq h_1, h_2 } } \sum_{ m' \neq m } \prod_{ i = 3 }^4 \Big| \Big[ \frac{ h_i - 1 }{ N }, \frac{ h_i }{ N } \Big) \cap [ \bar{ g }_k^{ ( m' - 1 ) }, \bar{ g }_k^{ ( m' ) } ) \Big| \Bigg]. \label{eq:stratified_multiple_merger_bound}
\end{align}
Similarly to \eqref{eq:stratified_numerator}, we have the bounds
\begin{align*}
\sum_{ h_3 \neq h_1, h_2 } \Big| \Big[ \frac{ h_3 - 1 }{ N }, \frac{ h_3 }{ N } \Big) \cap [ \bar{ g }_k^{ ( m - 1 ) }, \bar{ g }_k^{ ( m ) } ) \Big| &\leq\frac{ \gamma }{ N }, \\
\sum_{ \substack{ ( h_3, h_4 ) \in [ N ]_d^2 \\ \text{both } \neq h_1, h_2 } } \sum_{ m' \neq m } \prod_{ i = 3 }^4 \Big| \Big[ \frac{ h_i - 1 }{ N }, \frac{ h_i }{ N } \Big) \cap [ \bar{ g }_k^{ ( m' - 1 ) }, \bar{ g }_k^{ ( m' ) } ) \Big| &\leq \frac{ \gamma^2 }{ N }.
\end{align*}
Substituting these into \eqref{eq:stratified_multiple_merger_bound} yields
\begin{align*}
&\bbP^{ \bfX }( | \bar{ G }_k^{ N, n } | < | \xi | - 1 | \bar{ G }_{ k - 1 }^{ N, n } = \xi, G_j^{ N, n } = ( \xi, \ell ) ) \\
&\leq \sum_{ \ell_1 \in [ N ]_d^{ | \xi | } } \bbP^{ \bfX }( G_{ k - 1 }^{ N, n } = ( \xi, \ell_1 ) | \bar{ G }_{ k - 1 }^{ N, n } = \xi, G_j^{ N, n } = ( \xi, \ell_0 ) ) \frac{ \gamma^2 }{ ( N )_3 } \Bigg[ \frac{ | \xi | - 2 }{ 3 } + \frac{ N \gamma^2 }{ 2 ( N - 3 ) } \binom{ | \xi | - 2 }{ 2 } \Bigg]\\
&\phantom{\leq} \times  \sum_{ ( v_1, v_2 ) \in [ | \xi | ]_{ d, u }^2 } \sum_{ ( h_1, h_2 ) \in [ N ]_d^2 } \sum_{ m = 1 }^N  \prod_{ i = 1 }^2 \frac{ | [ \frac{ h_i - 1 }{ N }, \frac{ h_i }{ N } ) \cap [ \bar{ g }_k^{ ( m - 1 ) }, \bar{ g }_k^{ ( m ) } ) | M_k( X_k( m ), X_{ k - 1 }( \ell_1^{ v_i } ) ) }{ \sum_{ r = 1 }^N | [ \frac{ h_i - 1 }{ N }, \frac{ h_i }{ N } ) \cap [ \bar{ g }_k^{ ( r - 1 ) }, \bar{ g }_k^{ ( r ) } ) | M_k( X_k( r ), X_{ k - 1 }( \ell_1^{ v_i } ) ) }.
\end{align*}
To obtain the required bound, the right-hand side must be bounded above by the product of $\bbP^{ \bfX }( | \bar{ G }_k^{ N, n } | < | \xi | | \bar{ G }_{ k - 1 }^{ N, n } = \xi, G_j^{ N, n } = ( \xi, \ell ) )$ and a factor tending to zero as $N \to \infty$.
Currently, the sum over $( v_1, v_2 ) \in [ | \xi | ]_{ d, u }^2$ overcounts mergers involving more than two lineages.
By \eqref{eq:m_mixing}, 
\begin{align*}
&\frac{ | [ \frac{ h_i - 1 }{ N }, \frac{ h_i }{ N } ) \cap [ \bar{ g }_k^{ ( m - 1 ) }, \bar{ g }_k^{ ( m ) } ) | M_k( X_k( m ), X_{ k - 1 }( \ell_1^{ v_i } ) ) }{ \sum_{ r = 1 }^N | [ \frac{ h_i - 1 }{ N }, \frac{ h_i }{ N } ) \cap [ \bar{ g }_k^{ ( r - 1 ) }, \bar{ g }_k^{ ( r ) } ) | M_k( X_k( r ), X_{ k - 1 }( \ell_1^{ v_i } ) ) } \\
&\leq \gamma^2 \frac{ | [ \frac{ h_i - 1 }{ N }, \frac{ h_i }{ N } ) \cap [ \bar{ g }_k^{ ( m - 1 ) }, \bar{ g }_k^{ ( m ) } ) | M_k( X_k( m ), X_{ k - 1 }( \ell_1^i ) ) }{ \sum_{ r = 1 }^N | [ \frac{ h_i - 1 }{ N }, \frac{ h_i }{ N } ) \cap [ \bar{ g }_k^{ ( r - 1 ) }, \bar{ g }_k^{ ( r ) } ) | M_k( X_k( r ), X_{ k - 1 }( \ell_1^i ) ) }
\end{align*}
for $v_i, i \in [ | \xi | ]$, so that
\begin{align*}
&\bbP^{ \bfX }( | \bar{ G }_k^{ N, n } | < | \xi | - 1 | \bar{ G }_{ k - 1 }^{ N, n } = \xi, G_j^{ N, n } = ( \xi, \ell ) ) \\
&\leq \sum_{ \ell_1 \in [ N ]_d^{ | \xi | } } \bbP^{ \bfX }( G_{ k - 1 }^{ N, n } = ( \xi, \ell_1 ) | \bar{ G }_{ k - 1 }^{ N, n } = \xi, G_j^{ N, n } = ( \xi, \ell_0 ) )  \frac{ \gamma^6 ( | \xi | - 2 ) }{ N - 2 } \Bigg[ \frac{ 1 }{ 3 } + \frac{ N \gamma^2 ( | \xi | - 3 ) }{ 4 ( N - 3 ) } \Bigg] \binom{ | \xi | }{ 2 } \\
&\phantom{\leq} \times \frac{ 1 }{ ( N )_2 } \sum_{ ( h_1, h_2 ) \in [ N ]_d^2 } \sum_{ m = 1 }^N \prod_{ i = 1 }^2 \frac{ | [ \frac{ h_i - 1 }{ N }, \frac{ h_i }{ N } ) \cap [ \bar{ g }_k^{ ( m - 1 ) }, \bar{ g }_k^{ ( m ) } ) | M_k( X_k( m ), X_{ k - 1 }( \ell_1^i ) ) }{ \sum_{ r = 1 }^N | [ \frac{ h_i - 1 }{ N }, \frac{ h_i }{ N } ) \cap [ \bar{ g }_k^{ ( r - 1 ) }, \bar{ g }_k^{ ( r ) } ) | M_k( X_k( r ), X_{ k - 1 }( \ell_1^i ) ) }.
\end{align*}
Now the third line on the right-hand side is the probability that lineages $\ell_1^1$ and $\ell_1^2$ merge, and hence a lower bound on the probability of at least one merger.
Thus
\begin{align*}
&\bbP^{ \bfX }( | \bar{ G }_k^{ N, n } | < | \xi | - 1 | \bar{ G }_{ k - 1 }^{ N, n } = \xi, G_j^{ N, n } = ( \xi, \ell ) ) \\
&\leq \frac{ \gamma^6 ( | \xi | - 2 ) }{ N - 2 } \Bigg[ \frac{ 1 }{ 3 } + \frac{ N \gamma^2 ( | \xi | - 3 ) }{ 4 ( N - 3 ) } \Bigg] \binom{ | \xi | }{ 2 } \bbP^{ \bfX }( | \bar{ G }_k^{ N, n } | < | \xi | | \bar{ G }_{ k - 1 }^{ N, n } = \xi, G_j^{ N, n } = ( \xi, \ell_0 ) ),
\end{align*}
and hence \eqref{eq:large_mergers_vanish} holds.

To show that \eqref{eq:binary_uniformity} holds, we begin by considering a transition between $G_{ k - 1 }^{ N, n } = ( \xi, \ell_0 )$ and $G_k^{ N, n } = ( \xi, \ell_1 )$ for an $\ell_1$ which satisfies $\min_{i, m \in [|\xi|]_d^2} | \ell_1^i - \ell_1^m| \geq \lceil \gamma \rceil$.
It is possible to construct such an $\ell_1$ because it takes values in $[N]_d^{|\xi|}$, sufficiently well-separated entries can be chosen from $[| \xi | ( \lceil \gamma \rceil + 1 )]_d^{|\xi|}$, and $[| \xi | ( \lceil \gamma \rceil + 1 )]_d^{|\xi|} \subset [N]_d^{|\xi|}$ for any sufficiently large $N$.
By \eqref{eq:m_mixing} and \eqref{eq:stratified_denominator}, the corresponding conditional transition probability can be bounded below:
\begin{align}
&\bbP( G_k^{ N, n } = ( \xi, \ell_1 ) | G_{ k - 1 }^{ N, n } = ( \xi, \ell_0 )  ) \notag \\
&= \frac{ 1 }{ ( N )_{ | \xi | } } \sum_{ ( h_1, \ldots, h_{ | \xi | } ) \in [ N ]_d^{ | \xi | } } \prod_{ i = 1 }^{ | \xi | } \frac{ | [ \frac{ h_i - 1 }{ N }, \frac{ h_i }{ N } ) \cap [ \bar{ g }_k^{ ( \ell_1^i - 1 ) }, \bar{ g }_k^{ ( \ell_1^i ) } ) | M_k( X_k( \ell_1^i ), X_{ k - 1 }( \ell_0^i ) ) }{ \sum_{ r = 1 }^N| [ \frac{ h_i - 1 }{ N }, \frac{ h_i }{ N } ) \cap [ \bar{ g }_k^{ ( r - 1 ) }, \bar{ g }_k^{ ( r ) } ) | M_k( X_k( r ), X_{ k - 1 }( \ell_0^i ) ) } \notag \\
&\geq \frac{ N^{ | \xi | } }{ ( N )_{ | \xi | } \gamma^{ | \xi | } } \sum_{ ( h_1, \ldots, h_{ | \xi | } ) \in [ N ]_d^{ | \xi | } } \prod_{ i = 1 }^{ | \xi | } \Big| \Big[ \frac{ h_i - 1 }{ N }, \frac{ h_i }{ N } \Big) \cap [ \bar{ g }_k^{ ( \ell_1^i - 1 ) }, \bar{ g }_k^{ ( \ell_1^i ) } ) \Big|. \label{eq:mixing_lower_bound}
\end{align}
Our aim is to show that \eqref{eq:mixing_lower_bound} can be minorised by a uniform measure on the set $\{ \ell \in [ N ]_d^{ | \xi | } : \min_{i, m \in [ | \xi | ]_d^2} | \ell^i - \ell^m | > \lceil \gamma \rceil \}$.

Because the separation between entries of $\ell_1$ is at least $\lceil \gamma \rceil$ and $| [ \bar{ g }_k^{ ( r - 1 ) }, \bar{ g }_k^{ ( r ) } ) | \leq \gamma / N$ by \eqref{eq:uniform_weight_bound}, all $| \xi |$ intervals of the form $[ \bar{ g }_k^{ ( \ell_1^i - 1 ) }, \bar{ g }_k^{ ( \ell_1^i ) } )$ intersect with non-overlapping sets of intervals of the form $[ \frac{ h_i - 1 }{ N }, \frac{ h_i }{ N } )$.
There must also be at least one $( h_1, \ldots, h_{ | \xi | } ) \in [ N ]_d^{ | \xi | }$ such that, for every $i \in [ N ]$,
\begin{equation*}
\Big| \Big[ \frac{ h_i - 1 }{ N }, \frac{ h_i }{ N } \Big) \cap [ \bar{ g }_k^{ ( \ell_1^i - 1 ) }, \bar{ g }_k^{ ( \ell_1^i ) } ) \Big| \geq \frac{ 1 }{ 2 N \gamma },
\end{equation*}
because $| [ \bar{ g }_k^{ ( \ell_1^i - 1 ) }, \bar{ g }_k^{ ( \ell_1^i ) } ) | \geq 1 / ( \gamma N )$ and an interval of width $1 / ( \gamma N )$ cannot overlap more than two consecutive intervals of width $1 / N$.
Taking one such vector $( h_1, \ldots, h_{ | \xi | } )$ and bounding all other summands in \eqref{eq:mixing_lower_bound} from below by zero, we obtain
\begin{equation*}
\bbP( G_k^{ N, n } = ( \xi, \ell_1 ) | G_{ k - 1 }^{ N, n } = ( \xi, \ell_0 )  ) \geq \frac{ N^{ | \xi | } }{ ( N )_{ | \xi | } \gamma^{ | \xi | } } \frac{ 1 }{ ( 2 N \gamma )^{ | \xi | } } = \frac{ 1 }{ ( 2 \gamma^2 )^{ | \xi | } } \frac{ 1 }{ ( N )_{ | \xi | } }.
\end{equation*}
The number of elements of $\{ \ell \in [ N ]_d^{ | \xi | } : \min_{i, m \in [ | \xi | ]_d^2} | \ell^i - \ell^m | > \lceil \gamma \rceil \}$ is at least $N ( N - 2 \lceil \gamma \rceil - 1 ) \ldots ( N - ( | \xi | - 1 ) [ 2 \lceil \gamma \rceil + 1 ] )$, which is the number of elements obtained by letting each entry exclude itself, as well as $\lceil \gamma \rceil$ neighbours on each side, without regard for the fact that these exclusion zones can overlap.
Hence
\begin{align}
\bbP( G_k^{ N, n } = ( \xi, \ell_1 ) | G_{ k - 1 }^{ N, n } = ( \xi, \ell_0 )  ) &\geq \frac{ N ( N - 2 \lceil \gamma \rceil - 1 ) \ldots ( N - ( | \xi | - 1 ) [ 2 \lceil \gamma \rceil + 1 ] ) }{ ( 2 \gamma^2 )^{ | \xi | }  ( N )_{ | \xi | } | \{ \ell \in [ N ]_d^{ | \xi | } : | \ell^i - \ell^m | > \lceil \gamma \rceil \} | } \notag \\
&= \frac{ 1 + o( 1 ) }{ ( 2 \gamma^2 )^{ | \xi | } } \frac{ 1 }{ | \{ \ell \in [ N ]_d^{ | \xi | } : | \ell^i - \ell^m | > \lceil \gamma \rceil \} | }. \label{eq:stratified_minorisation}
\end{align}
In view of \eqref{eq:stratified_minorisation}, we can construct a transition from $G_{ k - 1 }^{ N, n } = ( \xi, \ell_0 )$ by first sampling an independent $Y \sim \text{Ber}( ( 1 - \varepsilon ) / ( 2 \gamma^2 )^{ | \xi | } )$, for some sufficiently small $\varepsilon > 0$, as soon as $N$ is sufficiently large.
If $Y = 1$, we set $G_k^{ N, n } = ( \xi, \ell' )$, where $\ell' \sim U( \{ \ell \in [ N ]_d^{ | \xi | } : | \ell^i - \ell^m | > \lceil \gamma \rceil \} )$.
Otherwise, $G_k^{ N, n }$ is sampled from a non-uniform remainder distribution, existence of which is guaranteed by \eqref{eq:stratified_minorisation}.
The number of generations until $Y = 1$ is Geo$( ( 1 - \varepsilon ) / ( 2 \gamma^2 )^{ | \xi | } )$-distributed, and hence by \eqref{eq:stratified_merger_ub},
\begin{align*}
&\bbP^{ \bfX }(Y = 1 \text{ at least once between generations } j \text { and } \tau_N( \xi, \ell, j; t ) ) \\
&= 1 - \Bigg(1 - \frac{ 1 - \varepsilon }{ ( 2 \gamma^2 )^{ | \xi | } } \Bigg)^{ \lfloor (1 + o(1)) N t / \gamma^4 \rfloor - j } \to 1.
\end{align*}
almost surely as $N \to \infty$.
Hence \eqref{eq:binary_uniformity} holds by essentially the same argument used in the case of multinomial resampling; see \eqref{eq:label_mixing} and the paragraph immediately following it.

To verify \eqref{timescale_assumption} we consider the probability that the two lexicographically lowest blocks merge, which is a lower bound on the overall merger probability.
Using \eqref{eq:m_mixing} and \eqref{eq:stratified_denominator},
\begin{align}
&\bbP^{ \bfX }( | \bar{ G }_k^{ N, n } | < | \xi | | \bar{ G }_{ k - 1 }^{ N, n } = \xi, G_j^{ N, n } = ( \xi, \ell_0 ) ) \notag \\
&\geq \sum_{ \ell_1 \in [ N ]_d^{ | \xi | } } \bbP^{ \bfX }( G_{ k - 1 }^{ N, n } = ( \xi, \ell_1 ) | \bar{ G }_{ k - 1 }^{ N, n } = \xi, G_j^{ N, n } = ( \xi, \ell_0 ) ) \sum_{ ( h_1, h_2 ) \in [ N ]_d^2 }  \frac{ 1 }{ ( N )_2 } \notag \\
&\phantom{= \sum_{ \ell_1 \in [ N ]_d^{ | \xi | } } } \times \sum_{ m = 1 }^N \prod_{ i \in \{ 1, 2 \} } \frac{ | [ \frac{ h_i - 1 }{ N }, \frac{ h_i }{ N } ) \cap [ \bar{ g }_k^{ ( m - 1 ) }, \bar{ g }_k^{ ( m ) } ) | M_k( X_k( m ), X_{ k - 1 }( \ell_1^i ) ) }{ \sum_{ r = 1 }^N | [ \frac{ h_i - 1 }{ N }, \frac{ h_i }{ N } ) \cap [ \bar{ g }_k^{ ( r - 1 ) }, \bar{ g }_k^{ ( r ) } ) | M_k( X_k( r ), X_{ k - 1 }( \ell_1^i ) ) } \notag \\
&\geq \sum_{ \ell_1 \in [ N ]_d^{ | \xi | } } \bbP^{ \bfX }( G_{ k - 1 }^{ N, n } = ( \xi, \ell_1 ) | \bar{ G }_{ k - 1 }^{ N, n } = \xi, G_j^{ N, n } = ( \xi, \ell_0 ) ) \sum_{ ( h_1, h_2 ) \in [ N ]_d^2 }  \frac{ N^2 }{ ( N )_2 \gamma^2 } \notag \\
&\phantom{= \sum_{ \ell_1 \in [ N ]_d^{ | \xi | } } } \times \sum_{ m = 1 }^N \prod_{ i \in \{ 1, 2 \} } \Big| \Big[ \frac{ h_i - 1 }{ N }, \frac{ h_i }{ N } \Big) \cap \Big[ \bar{ g }_k^{ ( m - 1 ) }, \bar{ g }_k^{ ( m ) } \Big) \Big|. \label{stratified_lb_start}
\end{align}
Many of the terms of the sum over $(h_1, h_2) \in [N]_d^2$ on the right-hand side can evaluate to zero.
However, if $\varepsilon_N \leq 1/N$, then by \eqref{eq:environment} and with probability converging to one as $N \to \infty$, the particle weights are such that one of the two possibilities depicted in Figure \ref{fig:max_weight} occurs for some $h \in [N]$ and $m \in [N]$.
\begin{figure}[!ht]
\centering
	\begin{tikzpicture}
		\draw (0, 2.4) node[above]{$h - 1$} -- (0, 2) -- (1.5, 2) -- (3, 2);
		\draw (5.5, 2) -- (10, 2) -- (10, 2.4) node[above]{$h + 1$};
		\draw(1.5, 2) -- (1.5, 2.4) node[above]{$h$};
		\draw(3, 2) -- (3, 2.4) node[above]{$h + 1$};
		\draw(5.5, 2) -- (5.5, 2.4) node[above]{$h - 2$};
		\draw(7, 2) -- (7, 2.4) node[above]{$h - 1$};
		\draw(8.5, 2) -- (8.5, 2.4) node[above]{$h$};
		\draw (0.4, 1) node[above]{$\bar{g}_k^{(m - 1)}$} -- (0.4, 0.6) -- (2.5, 0.6) -- (2.5, 1) node[above]{$\bar{g}_k^{(m)}$};
		\draw (6.5, 1) node[above]{$\bar{g}_k^{(m - 1)}$} -- (6.5, 0.6) -- (9.25, 0.6) -- (9.25, 1) node[above]{$\bar{g}_k^{(m)}$};
	\end{tikzpicture}
\caption{Two ways in which a normalised particle weight taking values between $[1/N, 2/N]$ can overlap stratification intervals.}
\label{fig:max_weight}
\end{figure}
\noindent
For a given length $g_k^{(m)} - g_k^{(m - 1)}$, the product of overlaps in the left panel of Figure \ref{fig:max_weight} is minimised by aligning $g_k^{(m)}$ with $(h + 1) / N$, or equivalently $g_k^{(m - 1)}$ with $(h - 1) / N$, yielding
\begin{align}
\Big| \Big[ \frac{ h - 1 }{ N }, \frac{ h }{ N } \Big) \cap \Big[ \bar{ g }_k^{ ( m - 1 ) }, \bar{ g }_k^{ ( m ) } \Big) \Big| &\Big| \Big[ \frac{ h }{ N }, \frac{ h + 1 }{ N } \Big) \cap \Big[ \bar{ g }_k^{ ( m - 1 ) }, \bar{ g }_k^{ ( m ) } \Big) \Big| \notag \\
&\geq \Big(g_k^{(m)} - g_k^{(m - 1)} - \frac{1}{N}\Big) \frac{1}{N} \geq \frac{\varepsilon_N}{N}. \label{max_weight_bound_1}
\end{align}
The right panel yields the same bound after summing up the two successive pairs $(h - 2, h - 1)$ and $(h - 1, h)$:
\begin{align}
&\Big| \Big[ \frac{ h - 2 }{ N }, \frac{ h - 1 }{ N } \Big) \cap \Big[ \bar{ g }_k^{ ( m - 1 ) }, \bar{ g }_k^{ ( m ) } \Big) \Big| \Big| \Big[ \frac{ h - 1 }{ N }, \frac{ h }{ N } \Big) \cap \Big[ \bar{ g }_k^{ ( m - 1 ) }, \bar{ g }_k^{ ( m ) } \Big) \Big| \notag \\
&\phantom{\geq}+ \Big| \Big[ \frac{ h - 1 }{ N }, \frac{ h }{ N } \Big) \cap \Big[ \bar{ g }_k^{ ( m - 1 ) }, \bar{ g }_k^{ ( m ) } \Big) \Big| \Big| \Big[ \frac{ h }{ N }, \frac{ h + 1 }{ N } \Big) \cap \Big[ \bar{ g }_k^{ ( m - 1 ) }, \bar{ g }_k^{ ( m ) } \Big) \Big| \notag \\
&= \Big(g_k^{(m)} - g_k^{(m - 1)} - \frac{1}{N}\Big) \frac{1}{N} \geq \frac{\varepsilon_N}{N}. \label{max_weight_bound_2}
\end{align}
Moreover, the inequalities on the right-hand sides of \eqref{max_weight_bound_1} and \eqref{max_weight_bound_2} continue to hold if $\varepsilon_N > 1/N$, albeit with more terms to sum on the left-hand sides.
Substituting both into \eqref{stratified_lb_start} yields
\begin{align}
&\bbP^{ \bfX }( | \bar{ G }_k^{ N, n } | < | \xi | | \bar{ G }_{ k - 1 }^{ N, n } = \xi, G_j^{ N, n } = ( \xi, \ell_0 ) ) \notag \\
&\geq \sum_{ \ell_1 \in [ N ]_d^{ | \xi | } } \bbP^{ \bfX }( G_{ k - 1 }^{ N, n } = ( \xi, \ell_1 ) | \bar{ G }_{ k - 1 }^{ N, n } = \xi, G_j^{ N, n } = ( \xi, \ell_0 ) ) \frac{ N^2 }{ ( N )_2 \gamma^2 } \frac{\varepsilon_N}{N} \notag \\
&= \frac{ 1 + o( 1 ) }{ \gamma^2 } \frac{\varepsilon_N}{N}, \label{stratified_lb}
\end{align}
which is bounded away from zero in $k$ for each fixed $N$.
Hence, for a fixed $t \in (0, \infty)$, we have $\bbP( \tau_N( \xi, \ell, j; t ) = \infty ) \to 0$ as $N \to \infty$ as required.

Finally, the bounds on the timescale $\tau_N(t)$ in \eqref{stratified_timescale_lb} and \eqref{stratified_timescale_ub} follow from \eqref{eq:stratified_merger_ub} and \eqref{stratified_lb}, as well as \eqref{binomial_sum}.
\end{proof}

\begin{proof}[Proof of Proposition \ref{prop:environment}]
For given weights $(g_k(X_k(1)), \ldots, g_k(X_k(N)))$, we define coefficients $(s_1, \ldots, s_N)$ via
\begin{equation*}
\frac{g_k(X_k(i))}{g_k(X_k(1)) + \ldots + g_k(X_k(N))} =: \frac{1}{N} + s_i.
\end{equation*}
Since ESS$(k) \leq N (1 - \eta_N)$, we have
\begin{equation*}
\frac{1}{(1 - \eta_N) N} \leq \sum_{i = 1}^N \Big(\frac{1}{N} + s_i \Big)^2 = \frac{1}{N} + \sum_{i = 1}^N s_i^2,
\end{equation*}
where the last step follows from $s_1 + \ldots + s_N = 0$.
Rearranging, we obtain
\begin{equation*}
\frac{\eta_N}{N (1 - \eta_N)} \leq \sum_{i = 1}^N s_i^2 \leq N \max_{i \in [N]}\{s_i^2\}.
\end{equation*}
If the maximising $s_i$ on the right-hand side is positive, for that $s_i$ we have
\begin{equation*}
s_i \geq \frac{1}{N} \sqrt{\frac{\eta_N}{1 - \eta_N}}.
\end{equation*}
If it is negative, the remaining $N - 1$ coefficients have to compensate for it and there must exist an $s_i$ satisfying
\begin{equation*}
s_i \geq \frac{1}{N(N - 1)} \sqrt{\frac{\eta_N}{1 - \eta_N}}.
\end{equation*}
\end{proof}

\section{Technical lemmas}\label{app:lemmas}

\begin{lem} \label{lem:multinomial}
Suppose $x_1,\dots,x_N \geq0$, $\alpha \in \bbN$, $N \in \bbN$, and $N \geq \alpha \geq 2$.
Then
\begin{equation} \label{eq:multinomial}
\Bigg(\sum_{i=1}^N x_i\Bigg)^\alpha \leq \sum_{ ( i_1,\dots,i_\alpha ) \in [ N ]_d^{ \alpha }  }^N \prod_{ m = 1 }^\alpha x_{ i_m } + \binom{\alpha}{2}\sum_{ m = 1 }^N x_m^2 \Bigg(\sum_{i=1}^N x_i\Bigg)^{\alpha-2}.
\end{equation}
\end{lem}
\begin{proof}
By the multinomial theorem and a simple partition of the resulting sum into two subsets, we have (letting $||\cdot||$ denote the $L^1$ norm):
\begin{align*}
\Bigg(\sum_{i=1}^N x_i\Bigg)^\alpha &= \sum_{\bfk \in \bbN^N:\:\| \bfk\| = \alpha} \binom{\alpha}{\bfk} \prod_{m=1}^N x_m^{k_m}\\
&= \sum_{\substack{\bfk \in \bbN^N:\:\| \bfk\| = \alpha\\\max\{ k_1, \ldots, k_N \} = 1}} \binom{\alpha}{\bfk} \prod_{m=1}^N x_m^{k_m}
 + \sum_{\substack{\bfk \in \bbN^N:\:\| \bfk\| = \alpha\\\max\{ k_1, \ldots, k_N \} \geq 2}} \binom{\alpha}{\bfk} \prod_{m=1}^N x_m^{k_m}
\end{align*}
We then partition the right-hand side according to which index of $\bfk$ has $k_i \geq 2$, and write the resulting summand as $\bfk^{(i)}$ to emphasize the distinctive index:
\begin{align*}
\Bigg(\sum_{i=1}^N x_i\Bigg)^\alpha &= \sum_{ ( i_1,\dots,i_\alpha ) \in [ N ]_d^{ \alpha } } \prod_{m=1}^\alpha x_{i_m} + \sum_{\substack{\bfk \in \bbN^N:\:\| \bfk\| = \alpha\\\max\{ k_1, \ldots, k_N \} \geq 2}} \binom{\alpha}{\bfk} \prod_{m=1}^N x_m^{k_m}\\
&\leq \sum_{ ( i_1,\dots,i_\alpha ) \in [ N ]_d^{ \alpha } } \prod_{m=1}^\alpha x_{i_j} + \sum_{i=1}^N \sum_{\substack{\bfk^{(i)} \in \bbN^N:\:\| \bfk^{(i)}\| = \alpha\\ k^{(i)}_i \geq 2}} \binom{\alpha}{\bfk^{(i)}} \prod_{m=1}^N x_m^{k^{(i)}_m}.
\end{align*}
This is an inequality because the right-hand side overcounts vectors $\bfk$ with more than one entry larger than 2.
For example $\bfk = (2,1,0,2,1)$ will appear on the right-hand side as an entry in the sum over both $\bfk^{(1)}$ and $\bfk^{(4)}$.
We can extract the two guaranteed entries of $k_i^{(i)}$ out of the innermost sum by using the change of variable $\bfl^{(i)} := \bfk^{(i)} - 2\bfe_i$, with $\bfe_i$ the tuple whose elements are $0$ with the exception of the $i^{\textrm{th}}$ which is 1:
\begin{align*}
\Bigg(\sum_{i=1}^N x_i\Bigg)^\alpha &= \sum_{ ( i_1,\dots,i_\alpha ) \in [ N ]_d^{ \alpha } } \prod_{m=1}^\alpha x_{i_m} \\
&\phantom{=} + \sum_{i=1}^N \sum_{\bfl^{(i)} \in \bbN^N:\:\| \bfl^{(i)}\| = \alpha-2} \frac{\alpha(\alpha-1)}{(\ell^{(i)}_i + 2)(\ell^{(i)}_i+1)}\binom{\alpha-2}{\bfl^{(i)}} x_i^2 \prod_{m=1}^N x_m^{\ell^{(i)}_m}\\
&\leq \sum_{ ( i_1,\dots,i_\alpha ) \in [ N ]_d^{ \alpha } } \prod_{m=1}^\alpha x_{i_m} + \binom{\alpha}{2}\sum_{i=1}^N x_i^2 \sum_{\bfl^{(i)} \in \bbN^N:\:\| \bfl^{(i)}\| = \alpha-2} \binom{\alpha-2}{\bfl^{(i)}} \prod_{m=1}^N x_m^{\ell^{(i)}_m}\\
&= \sum_{ ( i_1,\dots,i_\alpha ) \in [ N ]_d^{ \alpha } } \prod_{j=1}^\alpha x_{i_j} + \binom{\alpha}{2}\sum_{i=1}^N x_i^2  \Bigg(\sum_{m=1}^N x_m\Bigg)^{\alpha-2}.
\end{align*}
\end{proof}

The following lemma is reminiscent of the bounds of \cite[pages 442--443]{mohle1998}.

\begin{lem}\label{moehle_lb_lemma}
Let $N \in \bbN$ and $\{ a(m, i) \}_{m, i \in [N]}$ be an array of non-negative coefficients with $\sum_{m = 1}^N a(m, i) = 1$ for each $i \in [ N ]$, i.e.\ it coincides with the elements of a left stochastic matrix.
Let $r \in \bbN$, $v \in [ N ]$, and $( \ell^1, \ldots, \ell^r ) \in [ N ]_d^r$ be fixed.
Then
\begin{equation}
\sum_{ \substack{ ( v_1, \ldots, v_r ) \in [ N ]_d^r \\ \text{every } v_i \neq v } } \prod_{ i \in [ r ] } a( v_i, \ell^i ) \geq 1 - \sum_{ i \in [ r ] } a( v, \ell^i ) - \sum_{ ( h, h' ) \in [ r ]_d^2 } \sum_{ v' \in [ N ] \setminus \{ v \} } a( v', \ell^h ) a( v', \ell^{ h' } ). \label{eq:lem2}
\end{equation}
\end{lem}
\begin{proof}
Because $\sum_{ m = 1 }^N a( m, \ell ) = 1$, we have
\begin{align}
1 &= \prod_{i\in [r]}\left(\sum_{m \in [N]} a(m,\ell^i)\right) = \sum_{ (v_1,\dots,v_r) \in [ N ]^r} \prod_{i\in [r]} a(v_i,\ell^i) \notag\\
&= \sum_{ \substack{ ( v_1, \ldots, v_r ) \in [ N ]_d^r \\ \text{every } v_i \neq v } } \prod_{i\in [r]} a(v_i,\ell^i) + \sum_{ \substack{ ( v_1, \ldots, v_r ) \in [ N ]^r\setminus [N]^r_d \\ \text{every } v_i \neq v }} \prod_{i\in [r]} a(v_i,\ell^i)   + \sum_{ \substack{ ( v_1, \ldots, v_r ) \in [ N ]^r \\ \exists i:\: v_i = v }} \prod_{i\in [r]} a(v_i,\ell^i). \label{eq:lem2-step1}
\end{align}
Now
\begin{align*}
\sum_{ \substack{ ( v_1, \ldots, v_r ) \in [ N ]^r \\ \exists i:\: v_i = v }}  \prod_{i\in [r]} a(v_i,\ell^i) &\leq \sum_{h\in[r]} a(v,\ell^h)\sum_{ (v_1,\dots,v_{h-1},v_{h+1},\dots,v_{r}) \in [N]^{r-1}}  \prod_{i\in [r]\setminus \{h\}} a(v_i,\ell^i) \\
&=  \sum_{h\in[r]} a(v,\ell^h) \prod_{i\in[r]\setminus\{h\}} \left(\sum_{m\in[N]} a(m,\ell^i) \right) \\
&= \sum_{h\in[r]} a(v,\ell^h), 
\end{align*}
where the inequality follows by overcounting summands for which more than one element of $(v_1,\dots,v_r)$ is equal to $v$, and the last equality again uses that $\sum_{ m = 1 }^N a( m, \ell ) = 1$.

Similarly, overcounting summands for which more than two elements of $(v_1,\dots,v_r)$ are equal to another entry, we find
\begin{align}
\lefteqn{\sum_{ \substack{ ( v_1, \ldots, v_r ) \in [ N ]^r\setminus [N]^r_d \\ \text{every } v_i \neq v }} \prod_{i\in [r]} a(v_i,\ell^i)} \notag\\
&\leq \sum_{(h,h')\in[r]_d^2} \sum_{v'\in[N]\setminus\{v\}} a(v',\ell^h)a(v',\ell^{h'}) \sum_{( v_i : i \in [r] \setminus \{h, h' \}) \in ([N]\setminus\{v\})^{r-2}} \prod_{k\in[r]\setminus\{h,h'\}} a(v_i,\ell^k) \notag\\
&\leq \sum_{(h,h')\in[r]_d^2} \sum_{v'\in[N]\setminus\{v\}} a(v',\ell^h)a(v',\ell^{h'}) \sum_{( v_i : i \in [r] \setminus \{h,h' \}) \in [N]^{r-2}} \prod_{k\in[r]\setminus\{h,h'\}} a(v_i,\ell^k) \notag\\
&= \sum_{(h,h')\in[r]_d^2} \sum_{v'\in[N]\setminus\{v\}} a(v',\ell^h)a(v',\ell^{h'}) \prod_{k\in[r]\setminus\{h,h'\}} \left(\sum_{m\in [N]} a(m,\ell^k)\right) \notag\\
&= \sum_{(h,h')\in[r]_d^2} \sum_{v'\in[N]\setminus\{v\}} a(v',\ell^h)a(v',\ell^{h'}), \label{eq:lem2-step3}
\end{align}
where again we use $\sum_{ m = 1 }^N a( m, \ell ) = 1$ in the last step.
Combining \eqref{eq:lem2-step1}--\eqref{eq:lem2-step3} yields \eqref{eq:lem2}.
\end{proof}

\section*{Open access and data sharing}

For the purpose of open access, the authors have applied a Creative Commons Attribution (CC BY) license to any Author Accepted Manuscript version arising from this submission.
Data sharing is not applicable to this article as no new data were generated or analyzed.

\section*{Acknowledgements}

We thank Janique Krasnowska for help with proof-reading earlier versions of this manuscript, as well as an anonymous referee for many constructive questions and suggestions.
JK acknowledges the support of the Engineering and Physical Sciences Council of the United Kingdom (EPSRC; Grant Number EP/V049208/1).
PJ acknowledges the support of the Engineering and Physical Sciences Council of the United Kingdom (EPSRC; Grant Number EP/Y028783/1).
AMJ acknowledges the support of the Engineering and Physical Sciences Council of the United Kingdom (EPSRC; Grant Numbers EP/R034710/1 and EP/Y014650/1).

\bibliographystyle{alpha}\bibliography{smc}  

\end{document}